\documentclass[11pt,reqno]{amsart}

\usepackage{shuffle}

 \usepackage[ascii]{inputenc}  

\usepackage{tikz}
\usepackage{amsmath,amsthm,amsfonts,amssymb,latexsym,mathrsfs,color,extarrows}
\usetikzlibrary{positioning, shapes}

\usepackage[pagebackref]{hyperref}

\usepackage{subcaption}
\usetikzlibrary{arrows.meta}

\newtheorem{Theorem}{Theorem}
\newtheorem{Corollary}[Theorem]{Corollary}
\newtheorem{Proposition}[Theorem]{Proposition}

\newtheorem{Remark}[Theorem]{Remark}

\newtheorem{Definition}[Theorem]{Definition}

\allowdisplaybreaks

\DeclareMathOperator{\des}{des}
\DeclareMathOperator{\ides}{ides}
\DeclareMathOperator{\Des}{Des}
\DeclareMathOperator{\IDes}{IDes}

\DeclareMathOperator{\inv}{inv}

\DeclareMathOperator{\RLmin}{RLmin}

\DeclareMathOperator{\MIS}{MIS}

\DeclareMathOperator{\maj}{maj}
\DeclareMathOperator{\imaj}{imaj}

\def\RS{\mathop{\rm RS}\nolimits}

\def\URL{\mathop{\rm U}\nolimits^{RL}}

\def\AndI{\mathop{\rm And}\nolimits^{I}}  
\def\AndII{\mathop{\rm And}\nolimits^{I\!I}} 

\title{Inverse descent statistic for  Andr\'e and simsun  permutations}

\author{Guo-Niu Han}
\address{I.R.M.A., UMR 7501, Universit\'e de Strasbourg et CNRS, 7 rue
Ren\'e Descartes, F-67084 Strasbourg, France}
\email{guoniu.han@unistra.fr}

\author{Kathy Q. Ji}
\address{ Center for Applied Mathematics and KL-AAGDM,
Tianjin University,
Tianjin 300072, P.R. China
}
\email{kathyji@tju.edu.cn}

\author{Huan Xiong}
\address{
 Institute for Advanced Study in Mathematics, 
   Harbin Institute of Technology,
   Heilongjiang 150001, P.R. China
}
\email{huan.xiong.math@gmail.com}

\makeatletter
\@namedef{subjclassname@2020}{%
	\textup{2020} Mathematics Subject Classification}
\makeatother

\date{2025/11/16}
    
\begin{document}
	\begin{abstract}
Simsun permutations, Andr\'e I permutations and Andr\'e II  permutations are three combinatorial models for Euler numbers.  It's known that
the descent statistic is equidistributed  
over the set of  Andr\'e I permutations and the set of simsun permutations. In this paper, we prove that
the trivariate statistic $(\ides, \des, \maj)$,  comprising the  inverse descent, descent,  and major index,
are equidistributed over these three sets. This result is  equivalent to showing that the inverse descent is equidistributed over these three sets that share the same tree shape. The proof of the equidistribution of  the inverse descent  over  the set of Andr\'e I permutations and the set of Andr\'e II  permutations with the same tree shape reduces to establishing new refinements of Stanley's shuffle theorem. 
	\end{abstract}
	\maketitle


\section{Introduction}

In the field of enumerative combinatorics, several kinds of permutations are counted by {\it Euler numbers}, such as {\it alternating permutations}, {\it Andr\'e I and II permutations}, and {\it simsun  permutations}. {\it Euler numbers}, denoted by $E_n$, are a sequence of integers that arise in the Taylor series expansions of $\sec(x)+\tan(x)$. Their combinatorial significance was cemented by the work of {Andr\'e} in the late 19th century \cite{andre1881permutations}. Andr\'e proved that $E_n$ counts the number of {\it alternating permutations} of length $n$ (see \cite{stanley2009survey}), which are permutations $\sigma = \sigma_1\sigma_2\ldots\sigma_n$ satisfying $\sigma_1>\sigma_2<\sigma_3>\cdots$. 

\medskip

Andr\'e permutations were first introduced by Foata and Sch\"utzenberger and further studied by Strehl \cite{Str74} and Foata and Strehl \cite{FSt74, FSt76}.   For clarity, we will work with permutations of length $n$ for which each permutation is a sequence of $n$ distinct integers not necessarily from 1 to $n$.  The empty word $e$ and any single-letter word are defined as both {\it Andr\'e I permutations} and {\it Andr\'e II permutations}.
For a permutation $\sigma=\sigma_1\sigma_2\cdots \sigma_n$ ($n\geq 2$) of length $n$, we decompose it as $\sigma=\tau\,\min(\sigma)\,\tau'$. Here $\sigma$ is the concatenation of a left factor~$\tau$, followed by the minimum letter  $\min(\sigma)$, and a right factor $\tau'$.  Then, $\sigma$ is called an {\it Andr\'e I permutation} (resp. {\it Andr\'e II permutation}) if both $\tau$ and $\tau'$ are   Andr\'e I permutations (resp. Andr\'e II permutations), and the maximum letter of the subword $\tau\tau'$ lies in $\tau'$ (resp. the minimum letter of $\tau\tau'$ lies in $\tau'$).

\medskip

The set of all Andr\'e I permutations on the set $[n]:=\{1,2,\ldots, n\}$ is denoted by $\AndI_{n}$ and the set of  Andr\'e II  permutations  on the set $[n]$ is denoted by  $\AndII_{n}$. This inductive definition immediately reveals a connection to the Euler numbers, as it can be shown that the number of Andr\'e I permutations and Andr\'e II permutations on the set $[n]$ are equal, i.e., $E_n=|\AndI_n| = |\AndII_n|$.

\smallskip
Andr\'e I permutations for $n\leq 5$ are listed below:

$n=1$:\quad 1;\qquad $n=2$:\quad 12;\qquad
$n=3$:\quad 123, 213;

$n=4$:\quad 1234, 1324, 2314, 2134, 3124;

$n=5$:\quad 12345, 12435, 13425, 23415, 13245, 14235, 34125,
24135,\hfil\break
\indent\hphantom{$n=5$:\quad}23145, 21345, 41235, 31245, 21435, 32415, 41325, 31425.

\smallskip
Andr\'e II permutations for $n\leq 5$ are listed below:

$n=1$:\quad 1;\qquad $n=2$:\quad 12;\qquad
$n=3$:\quad 123, 312;

$n=4$:\quad 1234, 1423, 3412, 4123, 3124;

$n=5$:\quad 12345, 12534, 14523, 34512, 15234, 14235, 34125,
45123,\hfil\break
\indent\hphantom{$n=5$:\quad}35124, 51234, 41235, 31245, 51423, 53412, 41523, 31524.

\medskip

 Simsun  permutations were introduced by 
 Rodica Simion and Sheila Sundaram in a series of studies of homology representations of the symmetric group \cite{sundaram1995homology,sundaram1996plethysm}. To better elaborate on our results,  we adopt the following definition of simsun  permutations.  A permutation $\sigma = \sigma_1\sigma_2\ldots\sigma_n$ on the set $[n]$ is called a simsun  permutation if $\sigma_n=n$ and it contains no double descents, and this property is preserved after removing the elements $n, n-1, n-2, \ldots, 1$ in order. 
For example, it is easy to see that $\sigma= 21473658$ is a simsun  permutation since $21473658$, $2147365$, $214365$, $21435$, $2143$, $213$, $21$, $1$ have no double descents. Recall that  an index $i$ (where $1 \le i < n$)  is called a {\it descent} of a permutation $\sigma=\sigma_1\ldots \sigma_n$ if $\sigma_i > \sigma_{i+1}$ and an  index $i$ (where $1 \le i \leq n-2$)  is called a {\it double descent} if $\sigma_i > \sigma_{i+1}>\sigma_{i+2}$.

 Notably, if one removes the last element from a simsun  permutation  as defined here, the resulting permutation aligns with the original definition of simsun  permutations due to Simion and Sundaram.

The set of all simsun  permutations on the set $[n]$ is denoted by $\RS_{n}$.  A remarkable property of simsun  permutations is that  $|\RS_n|=E_{n}$. The notation $\RS_{n}$ was first adopted by Chow 
and Shiu  \cite{chow2011counting}.

Simsun  permutations for $n\leq 5$ are listed below:

$n=1$:\quad 1;\qquad $n=2$:\quad 12; \qquad $n=3$:\quad 123, 213;

$n=4$:\quad 1234, 1324, 2134, 2314, 3124;

 $n=5$:\quad 12345, 12435, 13245, 13425, 14235, 21345, 21435, 23145,\hfil\break
\indent\hphantom{$n=5$:\quad}23415, 24135, 31245, 31425, 34125, 41235, 41325, 42315. 
 
\medskip

Andr\'e permutations and simsun  permutations provide  new combinatorial interpretations for the Euler numbers. 
They  play an important role in the study of $cd$-indices of simplicial Eulerian posets. For results along this line, please see \cite{bayer2019cd,bayer1991new,He96, HR98, karu2006cd,purtill1993andre,stanley1994flag}. 
Other properties about Andr\'e permutations and simsun  permutations have been extensively studied by Barnabei et al. \cite{barnabei2020permutations}, Chow and Shiu \cite{chow2011counting}, Deutsch-Elizalde  \cite{deutsch2012restricted},  Disanto \cite{Di14},  Foata and the first author  \cite{FH01} and so on.     In particular, by constructing a bijection between the set of  Andr\'e I permutations and the set of simsun  permutations,  Chow and Shiu \cite{chow2011counting} observed that the number of descents are equidistributed over  Andr\'e I permutations and  simsun  permutations. Specifically, let $\des(\sigma)$ denote the number of descents of $\sigma$, they showed that 
$$
	 \sum_{\sigma \in \AndI_n}   t^{\des(\sigma)} =\sum_{\sigma \in \RS_{n}}  t^{\des(\sigma)}.
$$

In this paper, we show that the number of inverse descents are also equidistributed over  Andr\'e permutations and  simsun  permutations.  The number of inverse descents of a permutation $\sigma$ is simply the number of descents of its inverse permutation $\sigma^{-1}$, namely, $
\mathrm{ides}(\sigma) = \mathrm{des}(\sigma^{-1}).$ In fact,  we show that the trivariate statistic ($\ides$, $\des$, $\maj$)  are equidistributed over  Andr\'e permutations and  simsun  permutations, where  {\it  the major index} $\maj(\sigma)$ of $\sigma$ is defined  to be the sum of its descents of $\sigma$.    For brevity,   we adopt the notation 
   $n$-Andr\'e permutations for Andr\'e permutations  on   $[n]$ and $n$-simsun  permutations for   simsun permutations on   $[n]$.
\medskip 

Our main result is as follows. 

\begin{Theorem}\label{th:main} 
The trivariate statistic {\rm (}$\ides$, $\des$, $\maj${\rm )}  are equidistributed over  the set of $n$-Andr\'e I permutations,  $n$-Andr\'e II permutations and $n$-simsun permutations, i.e.,
   \begin{align*}
\sum_{\sigma \in \AndI_n} s^{\ides(\sigma)}t^{\des(\sigma)}  q^{\maj(\sigma)}
	&=\sum_{\sigma \in \AndII_n} s^{\ides(\sigma)}t^{\des(\sigma)}  q^{\maj(\sigma)}\\[5pt]
    &=\sum_{\sigma \in \RS_{n}} s^{\ides(\sigma)}t^{\des(\sigma)} q^{\maj(\sigma)}.
\end{align*}
   
\end{Theorem}

\medskip

To our knowledge, even the special case of the above result  for the univariate statistic ``$\ides$" is new:
$$
	A_n(s):=	\sum_{\sigma \in \AndI_n} s^{\ides(\sigma)}  
	=\sum_{\sigma \in \AndII_n} s^{\ides(\sigma)} 
    =\sum_{\sigma \in \RS_{n}} s^{\ides(\sigma)}.
$$
We list the first values of the polynomials $A_n(s)$ below:
\begin{align*}
A_1(s)&=1, \quad A_2(s)=1,\quad A_3(s)=s + 1,\quad
A_4(s)=4s + 1,\\
A_5(s)&=4s^2 + 11s + 1,\quad
A_6(s)=2s^3 + 32s^2 + 26s + 1.
\end{align*}

\medskip 

The proof of Theorem \ref{th:main} can be sketched as follows: By   sending Andr\'e permutations and simsun  permutations to increasing binary trees and applying Proposition \ref{th:equi-shape-des} in Section 2,  the proof of Theorem \ref{th:main}  reduces to showing that the  inverse descents (ides) is equidistributed  over the $n$-Andr\'e I permutations, the $n$-Andr\'e II permutations
and the $n$-simsun  permutations that share the same tree shape (see Theorem \ref{th:equi-shape} in Section 2). The proof of Theorem \ref{th:equi-shape} is split  into two parts: (a) proving the equidistribution of $\ides$  over 
 $n$-Andr\'e I  permutations and  $n$-Andr\'e II  permutations with the same tree shape (relation  (a) in Theorem \ref{th:equi-shape}),  and (b)   proving the equidistribution of $\ides$  over 
the  $n$-Andr\'e II  permutations and the simsun  permutations with the same tree shape  (relation  (b) in Theorem \ref{th:equi-shape}). Specially, the proof of relation  (a)   relies on an investigation of    the shuffle of permutations  (see Section 3 for details), while the proof of  relation (b) proceeds by   constructing a bijection between the set of  $n$-Andr\'e II permutations and the set of the $n$-simsun  permutations (see Section 4). It would be interesting to give a direct explicit bijective proof of relation  (a). 

\medskip 

In the proof of relation  (a) in Theorem \ref{th:equi-shape}, the following general result plays an important role. Note that this result applies to ordinary permutations, not solely Andr\'e permutations. 

Let $\mathfrak{S}_n$ denote the set of permutations on the set $[n]$. Suppose that $\sigma \in \mathfrak{S}_j$ and $\tau \in \mathfrak{S}_k$. We define the following three sets of permutations:
\begin{align*}
\sigma \lozenge\tau &=\{ \mu = \sigma'1\tau' \in \mathfrak{S}_{j+k+1} \mid
\sigma' \sim \sigma,\ \tau' \sim \tau \};\\
\sigma \vartriangle \tau 
&=\{ \mu = \sigma'1\tau' \in \mathfrak{S}_{j+k+1} \mid
\sigma' \sim \sigma,\ \tau' \sim \tau,\ j+k+1\in \tau' \};\\
\sigma \triangledown \tau
&=\{ \mu = \sigma'1\tau' \in \mathfrak{S}_{j+k+1} \mid
\sigma' \sim \sigma,\ \tau' \sim \tau,\ 2\in \tau' \},
\end{align*}
where $\sigma' \sim \sigma$ means that reducing the letters of  $\sigma'$ to $\{1,2,\ldots,j\}$ yields $\sigma$.

For example, if $\mu=\sigma'1\tau' =692581473$, then 
$\sigma'=69258$, which reduces to $\sigma=35124$, $\tau'=473$ which reduces to  $\tau=231$.

\begin{Theorem}\label{thm:ideseq:3}
Let  $\sigma \in \mathfrak{S}_j$ and $\tau\in \mathfrak{S}_k$
be two permutations with 
$\ides(\sigma)=j'$ and $\ides(\tau)=k'$. Then  
\begin{align} \label{eqn:ideseq:pfc}
\sum_{\mu \in \sigma \lozenge \tau} t^{\ides(\mu)}&=\sum_{i\geq 1} \binom{k-k'+j'+1}{i+j'-k'} \binom{j-j'+k'-1}{i-1} t^{i+j'},\\[5pt]
\sum_{\mu \in \sigma \vartriangle \tau} t^{\ides(\mu)}&=\sum_{i\geq 1} \binom{k-k'+j'}{i+j'-k'} \binom{j-j'+k'-1}{i-1} t^{i+j'}, \label{eqn:ideseq:pfccc} \\[5pt]
\sum_{\mu \in \sigma \triangledown \tau} t^{\ides(\mu)}&=\sum_{i\geq 1} \binom{k-k'+j'}{i+j'-k'} \binom{j-j'+k'-1}{i-1} t^{i+j'}. \label{eqn:ideseq:pfcc}
\end{align}

\end{Theorem}

Remark. When $t=1$, we derive  the following two identities, which also follow from the Chu-Vandermonde identity \cite[Example 1.1.17]{Stanley-EC1-2012}: 
\begin{align*}
\binom{j+k}{j} &=
\sum_{i\geq 1} \binom{j-d-1}{i-1} 
\binom{ k+d+1}{k-i+1} ,\\
\binom{j+k-1}{j} &=\sum_{i\geq 1}\binom{j-d-1}{i-1} 
\binom{ k+d}{k-i} .
\end{align*}

By considering the inverses of permutations, Theorem \ref{thm:ideseq:3} can be transformed to special cases of three refinements of Stanley's shuffle theorem (see  Theorem 
\ref{thm:shuffspe}). Stanley's shuffle theorem was first established by Stanley \cite{Stanley-1972} in his study of $P$-partitions. As observed by  Gessel and Zhuang \cite{Gessel-Zhuang-2018}, this theorem 
implies that  the major index (maj) and descent number (des)  are shuffle compatible, which has motivated several recent works, including those by  Adin, Gessel, Reiner  and Roichman \cite{Adin-Gessel-Reiner-Roichman-2021}, Baker-Jarvis and Sagan \cite{Baker-Sagan-2020}, Domagalski, Liang, Minnich and Sagan \cite{Domagalski-Liang-Minnich-2021}, Grinberg \cite{Grinberg-2018},  the second author and Zhang \cite{Ji-Zhang-2022} and Yang and Yan \cite{Yang-Yan-2022}. 
Bijective proofs of Stanley's  Shuffle Theorem have been given by  Goulden \cite{Goulden-1985}, the second author and Zhang \cite{Ji-Zhang-2024} and Stadler \cite{Stadler-1999}.  In particular,  the second author and Zhang \cite{Ji-Zhang-2024} established several refinements of this theorem based on their bijections. The proof of Theorem \ref{thm:ideseq:3} in this paper also relies on their bijection, see Section 4 for more details.

\bigskip

\section{Increasing binary trees}
 
In this section, we aim to demonstrate that the proof of  Theorem \ref{th:main}  is equivalent to proving Theorem \ref{th:equi-shape}  with the aid of the description of   Andr\'e permutations and simsun permutations   in terms of increasing binary trees.

\medskip 

A {\it binary tree} is a rooted tree in which every vertex has 
either (i) no children, (ii) a single left child, (iii) a single right child, or (iv) both a left child and a right child.
Vertices without children are called {\it leaves}, while all others are {\it internal} vertices. An {\it increasing binary tree} on the set $[n]$ is a binary tree with $n$ vertices labeled $1,2,\ldots, n$ such that  the labels along any path from the root are increasing. 

\medskip

It is well known that there exists a bijection $\Psi$ between the set of  permutations on $[n]$ and the set of increasing binary trees on $[n]$, see \cite[Chapter 1]{Stanley-EC1-2012}. More precisely, 
\begin{Definition}[The map $\Psi$] \label{defi:mappsi}
Let $\pi=\pi_1\pi_2\cdots \pi_n$  be a sequence of $n$ distinct letters not necessarily from $1$ to $n$. Define a binary tree $T_\pi$ as follows. If $\pi = \emptyset$, then $T_\pi = \emptyset$. If $\pi \neq \emptyset$, then let $i$ be the least letter of $\pi$. Thus $\pi$ can be factored uniquely in the form $\pi = \sigma i \tau$. Now let $i$ be the root of $T_\pi$, and let $T_\sigma$ and $T_{\tau}$ be the left and right subtrees obtained by removing $i$ (see Figure \ref{fig:inddef}). This yields an inductive definition of $T_\pi$.  
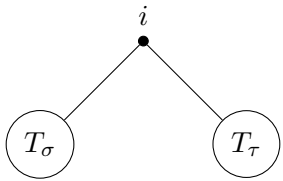
\begin{figure}[h]
    \centering
     \begin{tikzpicture}[
    scale=0.6,
    vertex/.style={shape=circle, draw, inner sep=1.3pt, fill=black},
    subtree/.style={shape=circle, draw},
    sibling distance=1cm,
    level distance=20mm,
    auto
]
\node[vertex, label=above:$i$] (i) at (0,0) {};
\node[subtree, below left=of i,xshift=0mm,yshift=-0mm] (left) {$T_\sigma$};
\draw (i) -- (left);
\node[subtree, below right=of i,xshift=0mm,yshift=-0mm] (right) {$T_{\tau}$};
\draw (i) -- (right);

\end{tikzpicture}
     \caption{An inductive definition of $T_\pi$} \label{fig:inddef}
\end{figure}
\end{Definition}

As observed by Foata and the first author \cite{FH01} and Chow and Shiu \cite{chow2011counting}, when restricted to Andr\'e  permutations and simsun  permutations,  the bijection $\Psi$  induces a
bijection sending Andr\'e permutations and simsun  permutations  to  special cases of increasing binary trees, which we refer to as Andr\'e trees and simsun  trees, respectively.   

Given  an increasing binary  tree $T$ and  a vertex $s$ of $T$, let  $T(s)$ denote the subtree of $T$ with the root $s$, and let $T_l(s)$ and $T_r(s)$ denote the left and right subtrees rooted at $s$, respectively. 

\begin{itemize}

\item An
increasing binary  tree $T$ is said to be an {\it   Andr\'e I tree} if for any internal vertex $s$, the right subtree
$T_r(s)$  contains the vertex of the maximum label in $T(s)$. By convention, the maxima  of an empty subtree is defined as $0$. 

\item An increasing binary tree $T$ is said to be an {\it  Andr\'e II tree} if for any internal vertex $s$, the right subtree
$T_r(s)$  contains the vertex with the minimum  label in $T(s)$ excluding $s$ itself. By convention, the  minima of an empty subtree is defined as $+\infty$. 

\item  An increasing binary tree on $[n]$ is called a {\it simsun  tree} if $n$ is its right-most vertex, and when the vertices $n, n-1, n-2, \ldots, 1$ are removed in sequence, the resulting trees $T'$ satisfy the following property: for any internal vertex $s$ in a left subtree of $T'$, if $T'_l(s)$ (the left subtree of $T'$ with the root $s$) is non-empty, then $T'_r(s)$ (the right subtree of $T'$ with the root $s$) is also non-empty.
\end{itemize}

Similarly, for brevity,   we adopt the notation 
   $n$-Andr\'e trees for Andr\'e trees  on   $[n]$ and $n$-simsun  trees for   simsun  trees on   $[n]$.

\begin{Proposition}[\cite{chow2011counting,FH01}]\label{Andreperincr2} There exists a bijection $\Psi$ between the set   of $n$-Andr\'e I permutations {\rm (}resp. $n$-Andr\'e II permutations, $n$-simsun 
permutations{\rm)}    and the set  of $n$-Andr\'e I trees {\rm (}resp. $n$-Andr\'e II trees, $n$-simsun  trees{\rm)}. 
\end{Proposition} 

 Fig. \ref{fig:bijAndI} depicts a bijection between the set of $4$-Andr\'e I permutations  and  the set of $4$-Andr\'e I trees, while  Fig. \ref{fig:bijAndII} shows a bijection between  the set of $4$-Andr\'e II permutations and the set of $4$-Andr\'e II trees.   Fig. \ref{fig:bijSim} illustrates a bijection between   the set of  $4$-simsun  permutations and the set of   $4$-simsun  trees.

\begin{figure}
    \centering
     \begin{tikzpicture}
[sibling distance=15mm, level distance=8mm,
every node/.style={circle, draw, fill=white, inner sep=0.5mm},
level 1/.style={sibling distance=15mm},
level 2/.style={sibling distance=15mm},
level 3/.style={sibling distance=15mm},
level 4/.style={sibling distance=15mm}]
\node [circle, draw, fill=white, inner sep=0.5mm] {1}
child {{} edge from parent [white]}
   child { node [circle, draw, fill=white, inner sep=0.5mm] {2}
   child {{} edge from parent [white]}
child { node [circle, draw, fill=white, inner sep=0.5mm] {3}
child {     edge from parent [white] }  child { node [circle, draw, fill=white, inner sep=0.5mm] {4}}}};
\node[draw=none] at (1,-3) {1~2~3~4};
\end{tikzpicture}\qquad
    \begin{tikzpicture}
[sibling distance=15mm, level distance=8mm,
every node/.style={circle, draw, fill=white, inner sep=0.5mm},
level 1/.style={sibling distance=15mm},
level 2/.style={sibling distance=15mm},
level 3/.style={sibling distance=15mm},
level 4/.style={sibling distance=15mm}]
\node [circle, draw, fill=white, inner sep=0.5mm] {1}
child {{} edge from parent [white]}
   child { node [circle, draw, fill=white, inner sep=0.5mm] {2}
child { node [circle, draw, fill=white, inner sep=0.5mm] {3}} child { node [circle, draw, fill=white, inner sep=0.5mm] {4}}};
\node[draw=none] at (1,-2.5) {1~3~2~4};
\end{tikzpicture}\qquad
\begin{tikzpicture}
[sibling distance=15mm, level distance=8mm,
every node/.style={circle, draw, fill=white, inner sep=0.5mm},
level 1/.style={sibling distance=15mm},
level 2/.style={sibling distance=15mm},
level 3/.style={sibling distance=15mm}]
\node [circle, draw, fill=white, inner sep=0.5mm] {1}
child { node [circle, draw, fill=white, inner sep=0.5mm] {2}
child {{} edge from parent [white]} 
child { node [circle, draw, fill=white, inner sep=0.5mm] {3}}}
child { node [circle, draw, fill=white, inner sep=0.5mm] {4}};
\node[draw=none] at (0,-2.5) {2~3~1~4};
\end{tikzpicture}\qquad \quad 
\begin{tikzpicture}
[sibling distance=15mm, level distance=8mm,
every node/.style={circle, draw, fill=white, inner sep=0.5mm},
level 1/.style={sibling distance=15mm},
level 2/.style={sibling distance=15mm},
level 3/.style={sibling distance=15mm},
level 4/.style={sibling distance=15mm}]
\node [circle, draw, fill=white, inner sep=0.5mm] {1}
child { node [circle, draw, fill=white, inner sep=0.5mm] {3}}
child { node [circle, draw, fill=white, inner sep=0.5mm] {2}
   child {{} edge from parent [white]}  child { node [circle, draw, fill=white, inner sep=0.5mm] {4}}};
   \node[draw=none] at (0,-2.3) {3~1~2~4};
\end{tikzpicture} \quad \quad 
\quad
\begin{tikzpicture}
[sibling distance=15mm, level distance=8mm,
every node/.style={circle, draw, fill=white, inner sep=0.5mm},
level 1/.style={sibling distance=15mm},
level 2/.style={sibling distance=15mm},
level 3/.style={sibling distance=15mm},
level 4/.style={sibling distance=15mm}]
\node [circle, draw, fill=white, inner sep=0.5mm] {1}
child { node [circle, draw, fill=white, inner sep=0.5mm] {2}}
child { node [circle, draw, fill=white, inner sep=0.5mm] {3}
   child {{} edge from parent [white]}  child { node [circle, draw, fill=white, inner sep=0.5mm] {4}}};
   \node[draw=none] at (0,-2.3) {2~1~3~4};
\end{tikzpicture}

\caption{The bijection between 4-Andr\'e I trees and 4-Andr\'e I permutations} \label{fig:bijAndI}
\end{figure}
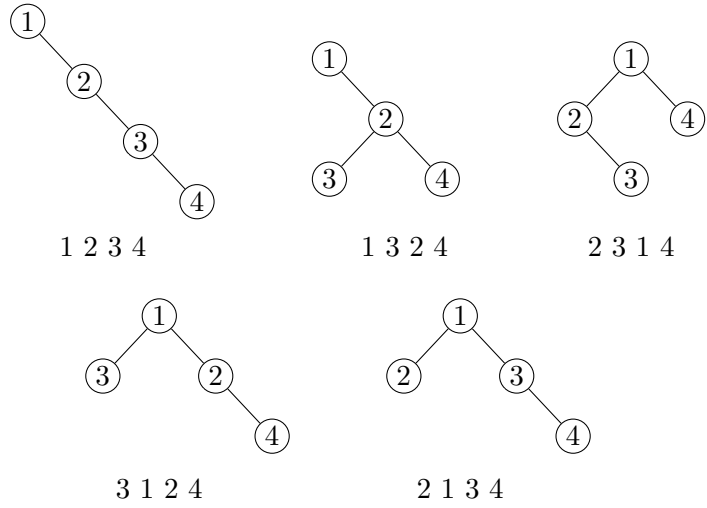
\begin{figure}
    \centering
\begin{tikzpicture}
[sibling distance=15mm, level distance=8mm,
every node/.style={circle, draw, fill=white, inner sep=0.5mm},
level 1/.style={sibling distance=15mm},
level 2/.style={sibling distance=15mm},
level 3/.style={sibling distance=15mm},
level 4/.style={sibling distance=15mm}]
\node [circle, draw, fill=white, inner sep=0.5mm] {1}
child {{} edge from parent [white]}
   child { node [circle, draw, fill=white, inner sep=0.5mm] {2}
   child {{} edge from parent [white]}
child { node [circle, draw, fill=white, inner sep=0.5mm] {3}
child {     edge from parent [white] }  child { node [circle, draw, fill=white, inner sep=0.5mm] {4}}}};
\node[draw=none] at (1,-3.2) {1~2~3~4};
\end{tikzpicture}\quad 
\begin{tikzpicture}
[sibling distance=15mm, level distance=8mm,
every node/.style={circle, draw, fill=white, inner sep=0.5mm},
level 1/.style={sibling distance=15mm},
level 2/.style={sibling distance=15mm},
level 3/.style={sibling distance=15mm},
level 4/.style={sibling distance=15mm}]
\node [circle, draw, fill=white, inner sep=0.5mm] {1}
child {{} edge from parent [white]}
   child { node [circle, draw, fill=white, inner sep=0.5mm] {2}
child { node [circle, draw, fill=white, inner sep=0.5mm] {4}} child { node [circle, draw, fill=white, inner sep=0.5mm] {3}}};
\node[draw=none] at (1,-2.5) {1~4~2~3};
\end{tikzpicture}\quad
\begin{tikzpicture}
[sibling distance=15mm, level distance=8mm,
every node/.style={circle, draw, fill=white, inner sep=0.5mm},
level 1/.style={sibling distance=15mm},
level 2/.style={sibling distance=15mm},
level 3/.style={sibling distance=15mm}]
\node [circle, draw, fill=white, inner sep=0.5mm] {1}
child { node [circle, draw, fill=white, inner sep=0.5mm] {3}
child {{} edge from parent [white]} 
child { node [circle, draw, fill=white, inner sep=0.5mm] {4}}}
child { node [circle, draw, fill=white, inner sep=0.5mm] {2}};
\node[draw=none] at (0,-2.5) {3~4~1~2};
\end{tikzpicture}\\
\begin{tikzpicture}
[sibling distance=15mm, level distance=8mm,
every node/.style={circle, draw, fill=white, inner sep=0.5mm},
level 1/.style={sibling distance=15mm},
level 2/.style={sibling distance=15mm},
level 3/.style={sibling distance=15mm},
level 4/.style={sibling distance=15mm}]
\node [circle, draw, fill=white, inner sep=0.5mm] {1}
child { node [circle, draw, fill=white, inner sep=0.5mm] {4}}
child { node [circle, draw, fill=white, inner sep=0.5mm] {2}
   child {{} edge from parent [white]}  child { node [circle, draw, fill=white, inner sep=0.5mm] {3}}};
   \node[draw=none] at (0,-2.3) {4~1~2~3};
\end{tikzpicture}\qquad
\begin{tikzpicture}
[sibling distance=15mm, level distance=8mm,
every node/.style={circle, draw, fill=white, inner sep=0.5mm},
level 1/.style={sibling distance=15mm},
level 2/.style={sibling distance=15mm},
level 3/.style={sibling distance=15mm},
level 4/.style={sibling distance=15mm}]
\node [circle, draw, fill=white, inner sep=0.5mm] {1}
child { node [circle, draw, fill=white, inner sep=0.5mm] {3}}
child { node [circle, draw, fill=white, inner sep=0.5mm] {2}
   child {{} edge from parent [white]}  child { node [circle, draw, fill=white, inner sep=0.5mm] {4}}};
   \node[draw=none] at (0,-2.3) {3~1~2~4};
\end{tikzpicture}
\caption{The bijection between 4-Andr\'e II trees and 4-Andr\'e II permutations} \label{fig:bijAndII}
\end{figure}
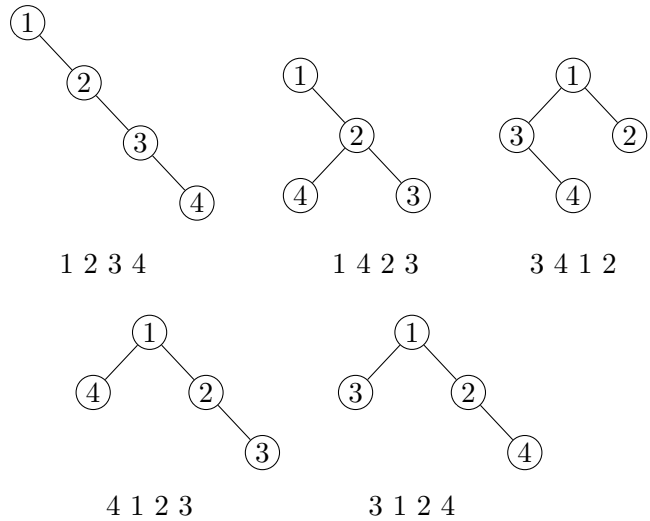

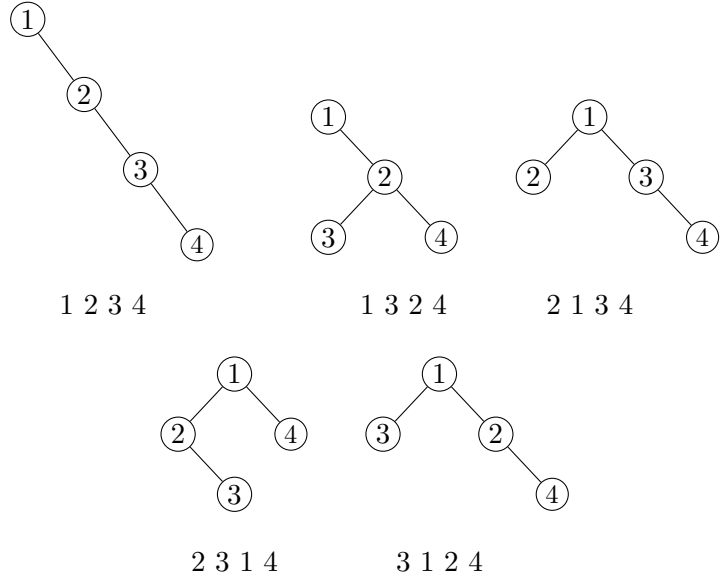
\begin{figure} 
    \centering
     \begin{tikzpicture}
[sibling distance=7mm, level distance=10mm,
every node/.style={circle, draw, fill=white, inner sep=0.5mm},
level 1/.style={sibling distance=15mm},
level 2/.style={sibling distance=15mm},
level 3/.style={sibling distance=15mm},
level 4/.style={sibling distance=15mm}]
\node [circle, draw, fill=white, inner sep=0.5mm] {1}
child {{} edge from parent [white]}
   child { node [circle, draw, fill=white, inner sep=0.5mm] {2}
   child {{} edge from parent [white]}
child { node [circle, draw, fill=white, inner sep=0.5mm] {3}
child {{} edge from parent [white]}
child { node [circle, draw, fill=white, inner sep=0.5mm] {\small{$4$}}}}}
;
\node[draw=none] at (1,-3.8) {1~2~3~4};
\end{tikzpicture}\qquad
    \begin{tikzpicture}
[sibling distance=15mm, level distance=8mm,
every node/.style={circle, draw, fill=white, inner sep=0.5mm},
level 1/.style={sibling distance=15mm},
level 2/.style={sibling distance=15mm},
level 3/.style={sibling distance=15mm},
level 4/.style={sibling distance=15mm}]
\node [circle, draw, fill=white, inner sep=0.5mm] {1}
child {{} edge from parent [white]}
   child { node [circle, draw, fill=white, inner sep=0.5mm] {2}
child { node [circle, draw, fill=white, inner sep=0.5mm] {3}} child {node [circle,draw, fill=white, inner sep=0.5mm]{\small{$4$}} }};
\node[draw=none] at (1,-2.5) {1~3~2~4};
\end{tikzpicture}\qquad 
\begin{tikzpicture}
[sibling distance=15mm, level distance=8mm,
every node/.style={circle, draw, fill=white, inner sep=0.5mm},
level 1/.style={sibling distance=15mm},
level 2/.style={sibling distance=15mm},
level 3/.style={sibling distance=15mm}]
\node [circle, draw, fill=white, inner sep=0.5mm] {1}
child { node [circle, draw, fill=white, inner sep=0.5mm] {2}}
child { node [circle, draw, fill=white, inner sep=0.5mm] {3}
child {{} edge from parent [white]}
child { node [circle, draw, fill=white, inner sep=0.5mm] {\small{$4$}}}};
\node[draw=none] at (0,-2.5) {2~1~3~4};
\end{tikzpicture}\\
\begin{tikzpicture}
[sibling distance=15mm, level distance=8mm,
every node/.style={circle, draw, fill=white, inner sep=0.5mm},
level 1/.style={sibling distance=15mm},
level 2/.style={sibling distance=15mm},
level 3/.style={sibling distance=15mm}]
\node [circle, draw, fill=white, inner sep=0.5mm] {1}
child { node [circle, draw, fill=white, inner sep=0.5mm] {2}
child {{} edge from parent [white]} 
child { node [circle, draw, fill=white, inner sep=0.5mm] {3}}}
child { node [circle, draw, fill=white, inner sep=0.5mm] {\small{${4}$}}};
\node[draw=none] at (0,-2.5) {2~3~1~4};
\end{tikzpicture}\qquad 
\begin{tikzpicture}
[sibling distance=15mm, level distance=8mm,
every node/.style={circle, draw, fill=white, inner sep=0.5mm},
level 1/.style={sibling distance=15mm},
level 2/.style={sibling distance=15mm},
level 3/.style={sibling distance=15mm},
level 4/.style={sibling distance=15mm}]
\node [circle, draw, fill=white, inner sep=0.5mm] {1}
child { node [circle, draw, fill=white, inner sep=0.5mm] {3}}
child { node [circle, draw, fill=white, inner sep=0.5mm] {2}child {{} edge from parent [white]}
child { node [circle, draw, fill=white, inner sep=0.5mm] {\small{$4$}}}};
   \node[draw=none] at (0,-2.5) {3~1~2~4};
\end{tikzpicture} 
\caption{The bijection between 4-simsun  trees and 4-simsun  permutations} \label{fig:bijSim}
\end{figure}

\medskip

The {\it  shape} of a labeled tree $T$ refers to its underlying unlabeled tree, denoted by ${\rm shape}(T)$. From the construction of the bijection $\Psi$, it is not difficult to see that the descent set of $\sigma $ is determined by the   shape of the tree corresponding to $\sigma$ under  the bijection $\Psi$. For a permutation $\sigma=\sigma_1\cdots \sigma_n$, its descent set is defined as  
\[\Des(\sigma)=\{1\leq i\leq n-1\colon \sigma_i > \sigma_{i+1}\}. 
\]
Then its descent number $\des(\sigma)$    is given by 
\begin{equation*}
\des(\sigma) := |\Des(\sigma)|.
\end{equation*}
and its  major index $\maj(\sigma)$   is given by 
\begin{equation*}
\maj(\sigma) := \sum_{i \in \Des(\sigma)} i.
\end{equation*}

\begin{Proposition} \label{th:equi-shape-des} Let $\sigma$ be a permutation, and let $T_\sigma=\Psi(\sigma)$ be the increasing binary tree corresponding to $\sigma$ under the bijection $\Psi$. Then the descent set $\Des(\sigma)$ of $\sigma$ is completely determined by  the shape of $T_\sigma$.   
\end{Proposition}

Note that the shape of $T_\sigma$ in Proposition \ref{th:equi-shape-des} is also referred to as the tree shape of the permutation $\sigma$.  Let $\URL_n$ denote the set of the rooted  unlabeled binary trees with $n$ vertices in which no internal vertex has only a left child.  It is not difficult to show that $|\URL_n|=M_n,$ where $M_n$ is the $n$-th Motzkin number defined by 
\[M(x)=1+xM(x)+x^2M(x)=\sum_{n\geq 0}M_n x^n.\]

From the definitions of Andr\'e  trees and simsun  trees, it is easy to verify that the tree shape of Andr\'e permutations and simsun permutations belong to the set $\URL_n$, that is, 

\begin{Proposition} \label{th:equi-shape-desb} Let $\sigma$ be a permutation, and let $T_\sigma=\Psi(\sigma)$ be the increasing binary tree corresponding to $\sigma$ under the bijection $\Psi$. Then  
\begin{align*}
\URL_n&=\{{\rm shape}(T_\sigma)\colon  \sigma \in \AndI_n \}\\
&= \{{\rm shape}(T_\sigma)\colon  \sigma \in \AndII_n \}\\
&= \{{\rm shape}(T_\sigma)\colon  \sigma \in \RS_n \}.
\end{align*}
\end{Proposition}

   Given an unlabeled binary tree $T$ in $\URL_n$, let $\AndI(T)$ (resp. $\AndII(T)$, $\RS(T)$) denote the set of $n$-Andr\'e I permutations  (resp. $n$-Andr\'e II permutations,  $n$-simsun permutations) with  tree shape $T$. Combining Proposition~\ref{th:equi-shape-des} and Proposition \ref{th:equi-shape-desb},  we see that the proof of Theorem \ref{th:main} is equivalent to  establishing the following result: 
 
\begin{Theorem}\label{th:equi-shape} For $n\geq 1$ and any  unlabeled binary tree $T \in \URL_n$, we have  
  \begin{align}
 	 \sum_{\sigma \in \AndI(T)}   s^{\ides(\sigma)} \overset{(a)}{=}\sum_{\sigma \in \AndII(T)}  s^{\ides(\sigma)} \overset{(b)}{=}\sum_{\sigma \in \RS(T)}  s^{\ides(\sigma)}.
\end{align}
\end{Theorem}

\medskip 

Figure \ref{fig:thmm} lists $8$-Andr\'e I permutations, $8$-Andr\'e II permutations   and $8$-simsun  permutations of the same tree shape $T$, all of which have $4$ inverse descents and share the descent set $\Des=\{1,4,6\}$.

\begin{figure}
\centering
\caption*{(a) $8$-Andr\'e I  permutations: }\label{fig:AndreI}
\vspace{-20pt}
\begin{tikzpicture}
[sibling distance=15mm, level distance=8mm,
every node/.style={circle, draw, fill=white, inner sep=0.5mm},
level 1/.style={sibling distance=17mm},
level 2/.style={sibling distance=14mm},
level 3/.style={sibling distance=8mm}]
\node [circle, draw, fill=white, inner sep=0.5mm] {1}
child { node [circle, draw, fill=white, inner sep=0.5mm] {2}
child {     edge from parent [white] } child {     edge from parent [white] } } child { node [circle, draw, fill=white, inner sep=0.5mm] {3}
child { node [circle, draw, fill=white, inner sep=0.5mm] {4}
child {     edge from parent [white] } child { node [circle, draw, fill=white, inner sep=0.5mm] {7}
child {     edge from parent [white] } child {     edge from parent [white] } } } child { node [circle, draw, fill=white, inner sep=0.5mm] {5}
child { node [circle, draw, fill=white, inner sep=0.5mm] {6}
child {     edge from parent [white] } child {     edge from parent [white] } } child { node [circle, draw, fill=white, inner sep=0.5mm] {8}
child {     edge from parent [white] } child {     edge from parent [white] } } } };
\node[draw=none,fill=none] at (0,0.8) {2\,1\,4\,7\,3\,6\,5\,8};
\end{tikzpicture}
\kern -4mm
\begin{tikzpicture}
[sibling distance=15mm, level distance=8mm,
every node/.style={circle, draw, fill=white, inner sep=0.5mm},
level 1/.style={sibling distance=17mm},
level 2/.style={sibling distance=14mm},
level 3/.style={sibling distance=8mm}]
\node [circle, draw, fill=white, inner sep=0.5mm] {1}
child { node [circle, draw, fill=white, inner sep=0.5mm] {4}
child {     edge from parent [white] } child {     edge from parent [white] } } child { node [circle, draw, fill=white, inner sep=0.5mm] {2}
child { node [circle, draw, fill=white, inner sep=0.5mm] {3}
child {     edge from parent [white] } child { node [circle, draw, fill=white, inner sep=0.5mm] {7}
child {     edge from parent [white] } child {     edge from parent [white] } } } child { node [circle, draw, fill=white, inner sep=0.5mm] {5}
child { node [circle, draw, fill=white, inner sep=0.5mm] {6}
child {     edge from parent [white] } child {     edge from parent [white] } } child { node [circle, draw, fill=white, inner sep=0.5mm] {8}
child {     edge from parent [white] } child {     edge from parent [white] } } } };
\node[draw=none,fill=none] at (0.3, 0.8) {4\,1\,3\,7\,2\,6\,5\,8}; \end{tikzpicture}
\kern -4mm
\begin{tikzpicture}
[sibling distance=15mm, level distance=8mm,
every node/.style={circle, draw, fill=white, inner sep=0.5mm},
level 1/.style={sibling distance=17mm},
level 2/.style={sibling distance=14mm},
level 3/.style={sibling distance=8mm}]
\node [circle, draw, fill=white, inner sep=0.5mm] {1}
child { node [circle, draw, fill=white, inner sep=0.5mm] {7}
child {     edge from parent [white] } child {     edge from parent [white] } } child { node [circle, draw, fill=white, inner sep=0.5mm] {2}
child { node [circle, draw, fill=white, inner sep=0.5mm] {3}
child {     edge from parent [white] } child { node [circle, draw, fill=white, inner sep=0.5mm] {6}
child {     edge from parent [white] } child {     edge from parent [white] } } } child { node [circle, draw, fill=white, inner sep=0.5mm] {4}
child { node [circle, draw, fill=white, inner sep=0.5mm] {5}
child {     edge from parent [white] } child {     edge from parent [white] } } child { node [circle, draw, fill=white, inner sep=0.5mm] {8}
child {     edge from parent [white] } child {     edge from parent [white] } } } };
\node[draw=none,fill=none] at (0.3, 0.8) {7\,1\,3\,6\,2\,5\,4\,8};
\end{tikzpicture}
\vspace{5pt} 
 
 \caption*{(b) $8$-Andr\'e II  permutations:}\label{fig:AndreII}
 \vspace{-20pt}
\begin{tikzpicture}
[sibling distance=15mm, level distance=8mm,
every node/.style={circle, draw, fill=white, inner sep=0.5mm},
level 1/.style={sibling distance=17mm},
level 2/.style={sibling distance=14mm},
level 3/.style={sibling distance=8mm}]
\node [circle, draw, fill=white, inner sep=0.5mm] {1}
child { node [circle, draw, fill=white, inner sep=0.5mm] {3}
child {     edge from parent [white] } child {     edge from parent [white] } } child { node [circle, draw, fill=white, inner sep=0.5mm] {2}
child { node [circle, draw, fill=white, inner sep=0.5mm] {5}
child {     edge from parent [white] } child { node [circle, draw, fill=white, inner sep=0.5mm] {8}
} } child { node [circle, draw, fill=white, inner sep=0.5mm] {4}
child { node [circle, draw, fill=white, inner sep=0.5mm] {7}
child {     edge from parent [white] } child {     edge from parent [white] } } child { node [circle, draw, fill=white, inner sep=0.5mm] {6}
child {     edge from parent [white] } child {     edge from parent [white] } } } };
\node[draw=none,fill=none] at (0.3,0.8) {3\,1\,5\,8\,2\,7\,4\,6};
\end{tikzpicture}
\kern -4mm
\begin{tikzpicture}
[sibling distance=15mm, level distance=8mm,
every node/.style={circle, draw, fill=white, inner sep=0.5mm},
level 1/.style={sibling distance=17mm},
level 2/.style={sibling distance=14mm},
level 3/.style={sibling distance=8mm}]
\node [circle, draw, fill=white, inner sep=0.5mm] {1}
child { node [circle, draw, fill=white, inner sep=0.5mm] {5}
child {     edge from parent [white] } child {     edge from parent [white] } } child { node [circle, draw, fill=white, inner sep=0.5mm] {2}
child { node [circle, draw, fill=white, inner sep=0.5mm] {4}
child {     edge from parent [white] } child { node [circle, draw, fill=white, inner sep=0.5mm] {8}
child {     edge from parent [white] } child {     edge from parent [white] } } } child { node [circle, draw, fill=white, inner sep=0.5mm] {3}
child { node [circle, draw, fill=white, inner sep=0.5mm] {7}
child {     edge from parent [white] } child {     edge from parent [white] } } child { node [circle, draw, fill=white, inner sep=0.5mm] {6}
child {     edge from parent [white] } child {     edge from parent [white] } } } };
\node[draw=none,fill=none] at (0.3,0.8) {5\,1\,4\,8\,2\,7\,3\,6};
\end{tikzpicture}
\kern -4mm
\begin{tikzpicture}
[sibling distance=15mm, level distance=8mm,
every node/.style={circle, draw, fill=white, inner sep=0.5mm},
level 1/.style={sibling distance=17mm},
level 2/.style={sibling distance=14mm},
level 3/.style={sibling distance=8mm}]
\node [circle, draw, fill=white, inner sep=0.5mm] {1}
child { node [circle, draw, fill=white, inner sep=0.5mm] {8}
child {     edge from parent [white] } child {     edge from parent [white] } } child { node [circle, draw, fill=white, inner sep=0.5mm] {2}
child { node [circle, draw, fill=white, inner sep=0.5mm] {4}
child {     edge from parent [white] } child { node [circle, draw, fill=white, inner sep=0.5mm] {7}
child {     edge from parent [white] } child {     edge from parent [white] } } } child { node [circle, draw, fill=white, inner sep=0.5mm] {3}
child { node [circle, draw, fill=white, inner sep=0.5mm] {6}
child {     edge from parent [white] } child {     edge from parent [white] } } child { node [circle, draw, fill=white, inner sep=0.5mm] {5}
child {     edge from parent [white] } child {     edge from parent [white] } } } };
\node[draw=none,fill=none] at (0.3,0.8) {8\,1\,4\,7\,2\,6\,3\,5};
\end{tikzpicture}
\vspace{5pt}

\caption*{(c) $8$-simsun  permutations:}\label{fig:simsun }
\vspace{-25pt}
\begin{tikzpicture}
[sibling distance=15mm, level distance=8mm,
every node/.style={circle, draw, fill=white, inner sep=0.5mm},
level 1/.style={sibling distance=17mm},
level 2/.style={sibling distance=14mm},
level 3/.style={sibling distance=8mm}]
\node [circle, draw, fill=white, inner sep=0.5mm] {1}
child { node [circle, draw, fill=white, inner sep=0.5mm] {2}
child {     edge from parent [white] } child {     edge from parent [white] } } child { node [circle, draw, fill=white, inner sep=0.5mm] {3}
child { node [circle, draw, fill=white, inner sep=0.5mm] {4}
child {     edge from parent [white] } child { node [circle, draw, fill=white, inner sep=0.5mm] {7}
 } } child { node [circle, draw, fill=white, inner sep=0.5mm] {5}
child { node [circle, draw, fill=white, inner sep=0.5mm] {6}
} child {   node [circle, draw, fill=white, inner sep=0.5mm] {8} } } };
\node[draw=none,fill=none] at (0.3,0.8) {\text{2\,1\,4\,7\,3\,6\,5\,8}};
\end{tikzpicture}
\kern -4mm
\begin{tikzpicture}
[sibling distance=15mm, level distance=8mm,
every node/.style={circle, draw, fill=white, inner sep=0.5mm},
level 1/.style={sibling distance=17mm},
level 2/.style={sibling distance=14mm},
level 3/.style={sibling distance=8mm}]
\node [circle, draw, fill=white, inner sep=0.5mm] {1}
child { node [circle, draw, fill=white, inner sep=0.5mm] {4}
child {     edge from parent [white] } child {     edge from parent [white] } } child { node [circle, draw, fill=white, inner sep=0.5mm] {2}
child { node [circle, draw, fill=white, inner sep=0.5mm] {3}
child {     edge from parent [white] } child { node [circle, draw, fill=white, inner sep=0.5mm] {7}
 } } child { node [circle, draw, fill=white, inner sep=0.5mm] {5}
child { node [circle, draw, fill=white, inner sep=0.5mm] {6}
} child {    node [circle, draw, fill=white, inner sep=0.5mm] {8} } } };
\node[draw=none, fill=none, inner sep=0pt, minimum size=0pt] at (0.3,0.8) {4\,1\,3\,7\,2\,6\,5\,8};
\end{tikzpicture}
\kern -4mm
\begin{tikzpicture}
[sibling distance=15mm, level distance=8mm,
every node/.style={circle, draw, fill=white, inner sep=0.5mm},
level 1/.style={sibling distance=17mm},
level 2/.style={sibling distance=14mm},
level 3/.style={sibling distance=8mm}]
\node [circle, draw, fill=white, inner sep=0.5mm] {1}
child { node [circle, draw, fill=white, inner sep=0.5mm] {7}
child {     edge from parent [white] } child {     edge from parent [white] } } child { node [circle, draw, fill=white, inner sep=0.5mm] {2}
child { node [circle, draw, fill=white, inner sep=0.5mm] {3}
child {     edge from parent [white] } child { node [circle, draw, fill=white, inner sep=0.5mm] {6}
 } } child { node [circle, draw, fill=white, inner sep=0.5mm] {4}
child { node [circle, draw, fill=white, inner sep=0.5mm] {5}
} child {     node [circle, draw, fill=white, inner sep=0.5mm] {$8$}  } } };
\node[draw=none,fill=none] at (0.3,0.8) {7\,1\,3\,6\,2\,5\,4\,8};
\end{tikzpicture}
\caption{$8$-Andr\'e I  permutations, $8$-Andr\'e I  permutations and $8$-simsun     permutations with $\Des=\{1,4,6\}$ and $\ides=4$} \label{fig:thmm}
\end{figure}
    \medskip 

The proof of Theorem \ref{th:equi-shape}  is divided into two parts, (a) and (b). It turns out that the proof of relation  (a)   reduces to  investigating the shuffle of permutations (see   Section 3) and the proof of relation  (b)   proceeds by   constructing a bijection between the set of  $n$-Andr\'e II permutations and the set of the $n$-simsun  permutations (see Section 4).

\section{A more general result on permutations}

The main objective of this section is to prove Theorem \ref{thm:ideseq:3}. Before proceeding, we first demonstrate how to derive    relation (a) in Theorem \ref{th:equi-shape}  using Theorem \ref{thm:ideseq:3}.  To this end, we begin by stating the following corollary, which is an immediate consequence of  Theorem \ref{thm:ideseq:3}. 

\begin{Corollary} \label{cor:ideseq:2}
Let $\sigma $ and $\hat{\sigma}$ be two permutations in $\mathfrak{S}_j$ such that $\ides(\sigma)=\ides(\hat{\sigma})=j'$, and let $\tau$ and $\hat{\tau}$ be two permutations in $\mathfrak{S}_k$ such that $\ides(\tau)=\ides(\hat{\tau})=k'$. Then
\begin{equation*}
\sum_{\mu \in \sigma \vartriangle  \tau}t^{\ides(\mu)} =\sum_{\mu \in \hat{\sigma} \triangledown   \hat{\tau}}t^{\ides(\mu)}. 
\end{equation*}
 
\end{Corollary}

\medskip

We are ready to establish   relation (a) in Theorem \ref{th:equi-shape} using Corollary~\ref{cor:ideseq:2}. 

\medskip 

\noindent{\it Proof of relation {\rm (}a{\rm )} in Theorem~ \ref{th:equi-shape}.} We proceed by induction on  $n$.   For $n=1$, relation  (a) clearly holds. Assume that it  holds for all   $p<n$. We aim to show that it also holds for $n$.  Let $T$ be an unlabeled (rooted) binary tree in $\URL_n$ with the left subtree $T^l$ and the right subtree $T^r$ of the root, respectively. 
By the definition of Andr\'e permutations, we have  
\[\AndI(T)=\bigcup_{\sigma \in \AndI(T^l)}  \bigcup_{ \tau \in \AndI(T^r)} \sigma \vartriangle  \tau \]
and 
\[\AndII(T)=\bigcup_{\hat{\sigma} \in \AndII(T^l)}  \bigcup_{ \hat{\tau} \in \AndII(T^l)} \hat{\sigma} \triangledown  \hat{\tau}. \]
This implies that 
\begin{align}
\sum_{\mu \in \AndI(T)}t^{\ides(\mu)}&=\sum_{\sigma \in \AndI(T^l) } \sum_{\tau \in  \AndI(T^r)}\sum_{\mu \in \sigma \vartriangle  \tau}t^{\ides(\mu)} \label{pf:lem7m1} \\[5pt]
\sum_{\mu \in  \AndII(T)}t^{\ides(\mu)}&=\sum_{\hat{\sigma} \in  \AndII(T^l) } \sum_{\hat{\tau} \in  \AndII(T^r)}\sum_{\mu \in \hat{\sigma}\triangledown   \hat{\tau}}t^{\ides(\mu)}.  \label{pf:lem7m2}
\end{align}
Note that  ${T}^l$ and $T^r$ are the left and right subtrees of the root of $T$, so their vertices are less than $n$.  
%
%
The induction hypothesis
implies that there exists a bijection $\phi^l$ between $\AndI({T}^l)$   and $\AndII({T}^l)$   such  that for $\sigma \in \AndI({T}^l)$   and $\phi(\sigma) \in \AndII({T}^l)$, we have $\ides(\sigma)=\ides(\phi^l(\sigma)).$ 
Similarly,  there exists a bijection $\phi^r$ between $\AndI({T}^r)$  and $\AndII({T}^r)$ such  that for $\tau \in \AndI({T}^r)$ and $\phi^r(\tau ) \in \AndII({T}^r)$, we have $\ides(\tau)=\ides(\phi^r(\tau )).$  
Hence by Corollary~\ref{cor:ideseq:2}, we arrive at  \begin{equation} \label{pf:lem7m3}
\sum_{\mu \in \sigma \vartriangle  \tau}t^{\ides(\mu )} =\sum_{\mu \in \phi^l(\sigma) \triangledown   \phi^r(\tau)}t^{\ides(\mu)}. 
\end{equation}
We therefore derive that 
\begin{align*}
\sum_{\mu \in \AndI(T)}t^{\ides(\mu)}&\stackrel{\eqref{pf:lem7m1}}{=}\sum_{\sigma \in \AndI(T^l) } \sum_{\tau \in \AndI(T^r)}\sum_{\mu  \in \sigma \vartriangle  \tau}t^{\ides(\mu)}   \\[5pt]
&\stackrel{\eqref{pf:lem7m3}}{=}\sum_{\phi^l(\sigma) \in \AndII(T^l) } \sum_{\phi^r(\tau) \in \AndII(T^r)}\sum_{\mu \in \phi^l(\sigma) \triangledown   \phi^r(\tau)}t^{\ides(\mu)} \\[5pt]
&\stackrel{\eqref{pf:lem7m2}}{=}\sum_{\mu  \in \AndII(T)}t^{\ides(\mu)}.
\end{align*}
This confirms that relation  (a) also holds for $n$. Thus, we complete the proof of  Theorem~ \ref{th:equi-shape} (a).  \qed 

\medskip 

We proceed to prove Theorem \ref{thm:ideseq:3}. As mentioned previously, the proof of Theorem  \ref{thm:ideseq:3} boils down to studying the shuffles of permutations.    Let $\mathcal{S}_n$ denote the set of permutations of length $n$, where a permutation is defined as a sequence of 
$n$ distinct integers (not necessarily restricted to $\{1,2,\ldots, n\}$).  Let $\pi \in \mathcal{S}_j$ and $\delta \in \mathcal{S}_k$ be two disjoint permutations, that is, permutations with no letters in common. We say that $\alpha \in \mathcal{S}_{j+k}$ is a shuffle of $\pi $ and $\delta$ if both $\pi$ and $\delta$ are subsequences of $\alpha$. The set of shuffles of $\pi $ and $\delta$ is denoted $\pi \shuffle \delta$.  For example, let $\pi=263$ and $\delta=14$, we have $
263\shuffle 14 = \{26314, 26134, 26143, 21463, 21634, 21643, 12463, 14263, 12634, 12643\}. $

\medskip 

 In his study of the theory of $P$-partitions,   Stanley \cite{Stanley-1972}  established the following formula for the joint statistic $(\des, \maj)$ over the set of permutation shuffles, which is referred to as Stanley's shuffle theorem.  Bijective proofs were later found by  Goulden \cite{Goulden-1985}, the second author and Zhang \cite{Ji-Zhang-2024} and Stadler \cite{Stadler-1999}. Recall that the Gaussian polynomial (also  called the $q$-binomial coefficients) is given by
\[{n \brack m}=\frac{(1-q^n)(1-q^{n-1})\cdots (1-q^{n-m+1})}{(1-q^m)(1-q^{m-1})\cdots (1-q)}.\] 
\begin{Theorem}[Stanley's Shuffle Theorem]\label{stanley}
Let $\pi \in \mathcal{S}_m$ and $\delta \in \mathcal{S}_n$ be two disjoint permutations, where ${\rm des}(\pi)=r$  and  ${\rm des}(\delta)=s$.  Then
\begin{equation*}
   \sum_{\alpha\in  \pi \shuffle \delta  \atop {\rm des}(\alpha)=k}q^{{\rm maj}(\alpha)}= {m-r+s \brack k-r} {n-s+r  \brack  k-s}  q^{{\rm maj}(\pi)+{\rm maj}(\delta)+(k-s)(k-r)}.
\end{equation*}
\end{Theorem}

To prove Theorem \ref{thm:ideseq:3},  it is necessary to introduce three special sets of shuffles, which are related to   the sets $\sigma \lozenge \tau $ (resp. $\sigma \vartriangle \tau$, $ \sigma \triangledown \tau$).   
Given two disjoint permutations $\pi=\pi_1\cdots \pi_m \in \mathcal{S}_m$ and $\delta=\delta_1\cdots \delta_n \in \mathcal{S}_n$, 

\begin{itemize}
\item  Let $\pi \shuffle_l \delta $ denote the set of shuffles   $\alpha=\alpha_1\cdots \alpha_{n+m}$ of  $\pi$ and $\delta$ such that $\alpha_1=\delta_1$.

\item   Let $\pi \shuffle_{ls} \delta $ denote the set of shuffles   $\alpha=\alpha_1\cdots \alpha_{n+m}$ of $\pi$ and $\delta$ such that $\alpha_1=\delta_1$ and $\alpha_{n+m}=\delta_n$.

\item Let $\pi \shuffle_{ll} \delta$ denote the set of shuffles   $\alpha=\alpha_1\cdots \alpha_{n+m}$ of $\pi$ and $\delta$ such that $\alpha_1=\delta_1$ and $\alpha_2=\delta_2$.

 \end{itemize}

 \medskip 
 
 For example, let $\pi=263$ and $\delta=14$, we have 
$263\shuffle_l 14 = \{12463, 14263,\break  12634, 12643\}$, $263\shuffle_{ls} 14 = \{12634\}$ and $263\shuffle_{ll} 14 = \{14263\}$.

\medskip

The following proposition establishes a connection between the set $\sigma \lozenge \tau $ (resp. $\sigma \vartriangle \tau$, $ \sigma \triangledown \tau$) and the three special sets of shuffles introduced above. 

\begin{Proposition}\label{idesdes} For $\sigma \in \mathfrak{S}_j$ and $\tau \in \mathfrak{S}_k$, let $\sigma^{-1}=\sigma^{-1}_1\cdots \sigma^{-1}_j$ and $\tau^{-1}=\tau^{-1}_1\cdots \tau^{-1}_k$  denote the inverses of $\sigma$ and $\tau$ respectively. Define $\pi=\sigma^{-1}_1\cdots \sigma^{-1}_j$ and $\delta=(j+1)(\tau^{-1}_1+j+1)\cdots (\tau^{-1}_k+j+1)$. 
Then
\begin{align}\label{eqn:ideseq:pfa}
\sum_{\mu \in \sigma \lozenge \tau} t^{\ides(\mu)}&=\sum_{\alpha \in \pi \shuffle_l \delta  } t^{\des(\alpha)},\\[5pt] \label{eqn:ideseq:pf2aaa}
\sum_{\mu \in \sigma \vartriangle \tau} t^{\ides(\mu)}&=\sum_{\alpha \in \pi \shuffle_{ls} \delta   } t^{\des(\alpha)},\\[5pt]
\sum_{\mu \in \sigma \triangledown \tau} t^{\ides(\mu)}&=\sum_{\alpha \in \pi \shuffle_{ll} \delta } t^{\des(\alpha)}. \label{eqn:ideseq:pf2aa}
\end{align}
\end{Proposition}

\begin{proof} Let $\mu=\sigma'  1  \tau' \in   \sigma \lozenge \tau$. By definition,  we see that
$\mu \in \mathfrak{S}_{j+k+1}$. Thus, we may write $\mu=\mu_1\cdots \mu_{j+k+1}$	 
 as a permutation of the set  $\{1,2,\ldots, j+k+1\}$. Note that $\mu_{j+1}=1$ and $\sigma' \cup \tau'=\{2,3, \ldots, j+k+1\}$.  Let 
 the elements of  $\sigma'$ be $i_1<i_2< \cdots< i_j$  and let those of  $\tau'$ be $j_1<j_2<\cdots<j_k$. 
 
 We now consider the inverse of $\mu$. Assume that $\mu^{-1}=\mu^{-1}_1\cdots \mu^{-1}_{j+k+1}$. It is clear to see  that   $\mu_1^{-1}=j+1$ and 
 \[\pi={\sigma}^{-1}=\mu_{i_1}^{-1}\mu_{i_2}^{-1}\cdots \mu_{i_{j}}^{-1}\]
 and 
 \[\delta:=(j+1)({\tau}^{-1}+{j+1})=\mu^{-1}_1\mu_{j_1}^{-1}\mu_{j_2}^{-1}\cdots \mu_{j_{k}}^{-1}.\] 
 Let $\alpha=\mu^{-1}$, 
Clearly,  $\alpha \in \pi \shuffle_l \delta$ and $\ides(\mu)=\des(\mu^{-1})=\des(\alpha)$. Moreover, this process is reversible. Thus, we prove \eqref{eqn:ideseq:pfa}.

For \eqref{eqn:ideseq:pf2aaa}, suppose that  $\mu=\sigma'  1  \tau' \in \sigma\vartriangle\tau$. By definition,  $j+k+1 \in \tau'$, so  $\mu^{-1}_{j+k+1} = \tau^{-1}_{k}+{j+1}=\delta_{k+1}$. Thus,   $\alpha \in \pi \shuffle_{ls} \delta$, establishing  \eqref{eqn:ideseq:pf2aaa}. 

Finally, if   $\mu=\sigma'  1  \tau' \in \sigma\triangledown \tau$, then $2 \in \tau'$, which implies   $\mu_{2}^{-1}= \tau^{-1}_1+{j+1}=\delta_2$. Hence,  $\alpha \in \pi \shuffle_{ll} \delta$, proving \eqref{eqn:ideseq:pf2aa}.  
\end{proof}

For example, given $
\sigma=3 5 1 2 4$ and   $\tau=2 3 1$, we see that $j=5$, $
\sigma^{-1}=3 4 1 5 2$ and $\tau^{-1}=3 1 2
$. Thus,  $\pi=3 4 1 5 2$ and $\delta=6978$. 
 
 (a) For $\mu=6 9 2 5 8 1 4 7 3 \in \sigma \lozenge\tau$, we see that $\mu^{-1}=6 3 9 7 4 1 8 5 2 \in \pi \shuffle_{l} \delta$.

  (b)  For $\mu=6 8 2 5 7 1 4 9 3 \in \sigma \vartriangle \tau$, we see that  
$\mu^{-1}=6 3 9 7 4 1 5 2 8 \in  \pi \shuffle_{ls} \delta$.

  (c)  For $\mu=6 9 3 5 8 1 4 7 2 \in \sigma \triangledown \tau$, we see that 
$\mu^{-1}=6 9 3 7 4 1 8 5 2 \in \pi \shuffle_{ll} \delta$.
 
\medskip 

With Proposition \ref{idesdes} at our disposal,   the proof of Theorem \ref{thm:ideseq:3}   comes down to establishing the following assertion.

\begin{Theorem} \label{thm:shuffspecor} Assume that  $\pi \in \mathcal{S}_j$ and $\delta \in \mathcal{S}_{k+1}$  are two  disjoint permutations, where $\des(\pi) = j'$ and $\des(\delta) = k'$. Moreover, $\delta_1<\delta_2$ and all of the elements of $\delta$ are larger than the elements of $\pi$. Then 
\begin{align*}  
\sum_{\alpha \in \pi \shuffle_l \delta  } t^{\des(\alpha)}&=\sum_{i\geq 1} \binom{k+1-k'+j'}{i+j'-k'} \binom{j-j'+k'-1}{i-1} t^{i+j'},\\[5pt]
\sum_{\alpha \in \pi \shuffle_{ls} \delta   } t^{\des(\alpha)}&=\sum_{i\geq 1} \binom{k-k'+j'}{i+j'-k'} \binom{j-j'+k'-1}{i-1} t^{i+j'},  \\[5pt]
\sum_{\alpha \in \pi \shuffle_{ll} \delta } t^{\des(\alpha)}&=\sum_{i\geq 1} \binom{k-k'+j'}{i+j'-k'} \binom{j-j'+k'-1}{i-1} t^{i+j'}. 
\end{align*} 

\end{Theorem}

\medskip

The following assertion can be viewed as  refinements of Stanley's shuffle theorem, from which Theorem \ref{thm:shuffspecor}   follows immediately by  letting $q\rightarrow 1$. 

\begin{Theorem} \label{thm:shuffspe} Assume that $\delta \in \mathcal{S}_m$ and  $\pi \in \mathcal{S}_n$  are two disjoint   permutations, where $\des(\delta) = r$ and $\des(\pi) = s$. Moreover, $\delta_1<\delta_2$ and all of the elements of $\delta$ are larger than the elements of $\pi$. Then  
\begin{align*}
  (1)&   \sum_{\substack{\alpha \in \pi \shuffle_l \delta  \\ \mathrm{des}(\alpha) = d}} q^{\mathrm{maj}(\alpha)}= {m-r+s \brack d-r} {n-s+r-1\brack d-s-1} \times q^{\mathrm{maj}(\delta) + \mathrm{maj}(\pi)  + (d - s)(d - r)}, \\[5pt]
    (2)&  \sum_{\substack{\alpha \in \pi \shuffle_{ls} \delta  \\ \mathrm{des}(\alpha) = d}} q^{\mathrm{maj}(\alpha)}= {m-r+s-1 \brack d-r} {n-s+r-1\brack d-s-1} \times q^{\mathrm{maj}(\delta) + \mathrm{maj}(\pi)  + (d - s)(d - r)}, \\[5pt]   
   (3) &   \sum_{\substack{\alpha \in \pi \shuffle_{ll} \delta  \\ \mathrm{des}(\alpha) = d}} q^{\mathrm{maj}(\alpha)}= {m-r+s-1 \brack d-r} {n-s+r-1\brack d-s-1} \times q^{\mathrm{maj}(\delta) + \mathrm{maj}(\pi)  + (d- s+1)(d - r)}.   
\end{align*}
\end{Theorem}

We conclude this section with a proof of the theorem. It turns out that the bijection used to establish Stanley's shuffle theorem by the second author and Zhang \cite{Ji-Zhang-2024} plays a crucial role. Let $\mathcal{P}_n(t,m)$ denote the set of partitions $\lambda=(\lambda_1,\ldots, \lambda_n)$ such that $\lambda_n \geq t$ and $\lambda_1 \leq m$. We have 
\begin{equation}\label{int-GassCoeft}
q^{nt}{n+m-t \brack n}=\sum_{\lambda \in \mathcal{P}_n(t,m)} q^{|\lambda|}.
\end{equation}

Thus,  Stanley's shuffle theorem \ref{stanley} is equivalent to the following statement. 

\begin{Theorem}\cite[Theorem 3.1]{Ji-Zhang-2024} \label{thm:insert}
Assume that $\delta \in \mathcal{S}_m$ and $\pi \in \mathcal{S}_n$ are two disjoint permutations, where $\des(\delta) = r$ and $\des(\pi) = s$. Let $\mathfrak{S}(\delta, \pi \vert d)$ denote the set of all shuffles of $\delta$ and $\pi$ with $d$ descents. Then there is a bijection $\Phi$ between $\mathfrak{S}(\delta, \pi \vert d)$ and $\mathcal{P}_{d-r}(d-s, m) \times \mathcal{P}_{n-d+r}(0, d-s)$, namely, for $\alpha \in \mathfrak{S}(\delta, \pi \vert d)$, we have $(\lambda, \mu) = \Phi(\alpha) \in \mathcal{P}_{d-r}(d-s, m) \times \mathcal{P}_{n-d+r}(0, d-s)$ such that
\[
\maj(\alpha) = \vert \lambda \vert + \vert \mu \vert + \maj(\delta) + \maj(\pi). 
\]
\end{Theorem}

The following map is a desired bijection  in Theorem \ref{thm:insert}, see \cite[Lemma 3.5 and Lemma 3.7]{Ji-Zhang-2024}. 

\begin{Definition} [The map $\Phi$] Let $\delta = \delta_1 \cdots \delta_m$ be a permutation with $r$ descents and let $\pi = \pi_1 \cdots \pi_n$ be a permutation with $s$ descents. Assume that $\alpha = \alpha_1 \cdots \alpha_{m+n}$ is the shuffle of $\delta$ and $\pi$ with $d$ descents. The pair of partitions $(\lambda, \mu) = \Phi(\alpha)$ can be constructed as follows: Let $\alpha^{(i)}$ denote the permutation obtained by removing $\pi_1, \pi_2, \ldots, \pi_i$ from $\alpha$. Obviously, $\alpha^{(n)} = \delta$. Here we assume that $\alpha^{(0)} = \alpha$. For $1 \leq i \leq n$, define
\[
t(i) = \maj(\alpha^{(i-1)}) - \maj(\alpha^{(i)}) - d_i(\pi),
\]
where $d_i(\pi)$ denotes the number of descents in $\pi$ greater than or equal to $i$. 

Since there are $d$ descents in $\alpha$ and there are $r$ descents in $\delta$, it follows that there exists $d-r$ permutations in $\alpha^{(1)}, \ldots, \alpha^{(n)}$, denoted by $\alpha^{(i_1)}, \ldots, \alpha^{(i_{d-r})}$ where $1 \leq i_1 < i_2 < \cdots < i_{d-r} \leq n$, such that $\des(\alpha^{(i_{l}-1)}) = \des(\alpha^{(i_l)}) + 1$ for $1 \leq l \leq d-r$. Let $\{j_1, \ldots, j_{n-d+r}\} \in \{1, \ldots, n\} \setminus \{i_1, i_2, \ldots, i_{d-r}\}$, where $1 \leq j_1 < j_2 < \cdots < j_{n-d+r} \leq n$. Then $\des(\alpha^{(j_{l}-1)}) = \des(\alpha^{(j_l)})$ for $1 \leq l \leq n-d+r$. The pair of partitions $(\lambda, \mu) = \Phi(\alpha)$ is defined by
\[
\lambda = \bigl(t(i_{d-r}), t(i_{d-r- 1}), \ldots, t(i_1)\bigr),  
\]
and
\[
\mu = \bigl(t(j_1), t(j_2), \ldots, t(j_{n-d+r})\bigr).  
\]
More precisely,   
 \[
m \geq t(i_{d-r}) \geq \cdots \geq t(i_1) \geq d-s \geq  t(j_1) \geq \cdots \geq t(j_{n-d+r}) \geq 0.  
\]
\end{Definition}

\medskip

Let $\sigma = \sigma_1 \cdots \sigma_n \in \mathcal{S}_n$ and $r \notin \sigma$. Recall that $\sigma^{(i)}(r)$ denotes the permutation obtained by inserting $r$ before $\sigma_{i+1}$ (or after $\sigma_i$ if $i = n$). For $0 \leq i \leq n$, define the major increment
\[
\operatorname{im}(\sigma, i, r) = \maj(\sigma^{(i)}(r)) - \maj(\sigma)
\]
and the major increment sequence
\[
\MIS(\sigma, r) = (\operatorname{im}(\sigma, 0, r), \ldots, \operatorname{im}(\sigma, n, r)).
\]
The first $i$ elements of $\MIS(\sigma,r)$ is defined by 
\[\MIS_i(\sigma,r)=({\rm im}(\sigma,0,r),\ldots,{\rm im}(\sigma,i-1,r)). 
\]

The proof of Theorem \ref{thm:insert} relies on the following two propositions, which are also essential for proving Theorem  \ref{thm:shuffspe}.

\begin{Proposition}\cite[Corollary 3.3]{Ji-Zhang-2024} \label{prop:insert}
Let $\sigma \in \mathcal{S}_n$ with $k$ descents and $r \notin \sigma$. Then 
$\MIS(\sigma, r)$ is a shuffling of $k + 1, k + 2, \ldots, n$ and $k, \ldots, 1, 0$. In particular,  if $\des(\sigma^{(i)}(r)) = \des(\sigma) + 1$, then
\[
\operatorname{im}(\sigma, i, r) = \max\{\operatorname{im}(\sigma, 0, r), \ldots, \operatorname{im}(\sigma, i - 1, r)\} + 1,
\]
otherwise,
\[
\operatorname{im}(\sigma, i, r) = \min\{\operatorname{im}(\sigma, 0, r), \ldots, \operatorname{im}(\sigma, i - 1, r)\} - 1.
\]
\end{Proposition}

\begin{Proposition} \cite[Proposition 3.4]{Ji-Zhang-2024}\label{prop:insertb}
Suppose that $\sigma$ is  a permutation of length $m$ with $r$ descents. Let $p,q\notin \sigma$ and let $\sigma^{(i-1)}(p)$ be the permutation by inserting $p$ before $\sigma_i$. Then 
$\MIS_i(\sigma^{(i-1)}(p),q)$ is a permutation of the set $\{{\rm im}(\sigma,j,p)+\chi(q>p)\mid 0\leq j< i\}$, where $\chi(T)=1$ if the statement $T$ is true and $\chi(T)=0$   otherwise.
\end{Proposition}

\medskip 

We are now in a position to prove
Theorem \ref{thm:shuffspe}. 

\medskip 

\noindent{\it Proof of Theorem \ref{thm:shuffspe}.}
(1) Let $\alpha \in \pi \shuffle_l \delta$, where $\des(\alpha)=d$. Assume that $\Phi(\alpha) = (\lambda, \mu)$. To prove (1) in this theorem, it is equivalent to show that 
\[
m \geq \lambda_1 \geq \cdots \geq \lambda_{d - r} \geq d - s> \mu_1 \geq \cdots \geq \mu_{n - d + r} \geq 0.
\]
From Theorem \ref{thm:insert}, we have
\[
m \geq \lambda_1 \geq \cdots \geq \lambda_{d - r} \geq d - s \geq \mu_1 \geq \cdots \geq \mu_{n - d + r} \geq 0.
\]

We proceed to show that $\mu_{1} < d-s$ if $\mu \neq \emptyset$.  Since $\alpha \in \pi \shuffle_l \delta$, we have $\alpha_1=\delta_1$.  
 Recall that $\alpha^{(1)}$ is the permutation obtained by removing $\pi_1$ from~$\alpha$. Then $\des(\alpha^{(1)})=d-1$ or $\des(\alpha^{(1)})=d$. 
 \medskip 
 
{\it Case 1.1}. If $\des(\alpha^{(1)})=d$, then $\des(\alpha^{(0)})=\des(\alpha^{(1)})=d$, and so $j_1=1$. Since $\alpha_1=\delta_1$, it implies that $\pi_1$ could not be inserted in the first position. Moreover,     $\operatorname{im}(\alpha^{(1)}, 0, \pi_1) = d $, so  by Proposition \ref{prop:insert},  we derive that
 \[\maj(\alpha^{(0)})-\maj(\alpha^{(1)})\leq d-1.\]
Thus, 
$\mu_1=t(1)=\maj(\alpha^{(0)})-\maj(\alpha^{(1)})-d_1(\pi)\leq d-1-s$.

 {\it Case 1.2}. If $\des(\alpha^{(1)})=d-1$, then by the definition of the map $\Phi$, we see that  $j_1>1$, and $\des(\alpha^{(j_1)})=d-j_1+1$. Similarly, $\pi_j$ could not be inserted in the first position of $\alpha^{(j_1)}$ and $\operatorname{im}(\alpha^{(j_1)}, 0, \pi_{j_1}) = d-j_1+1$, so 
 by Proposition \ref{prop:insert}, we derive that 
\[\maj(\alpha^{(j_1-1)})-\maj(\alpha^{(j_1)})\leq d-j_1.\]
By definition,   $d_{j_1}(\pi)\geq s-j_1+1$  so $\mu_1=t(j_1)=\maj(\alpha^{(j_1-1)})-\maj(\alpha^{(j_1)})-d_{j_1}(\pi)\leq d-j_1-s+j_1-1=d-s-1$. 
  
   \medskip

Conversely, let $\lambda \in \mathcal{P}_{d - r}(d - s, m)$ and $\mu \in \mathcal{P}_{n - d + r}(0, d - s-1)$. Assume that $\Phi^{-1}(\lambda, \mu) = \overline{\alpha}$, where the map $\Phi^{-1}$ is the inverse of $\Phi$. In light of Theorem \ref{thm:insert}, we derive that $\overline{\alpha} = \overline{\alpha}_1 \cdots \overline{\alpha}_{n + m}$ is a shuffle of $\delta$ and $\pi$ with $d$ descents. To prove that $\overline{\alpha} \in \pi \shuffle_l \delta$, it suffices to show that $\overline{\alpha}_{1} = \delta_1$. Suppose to the contrary that $\overline{\alpha}_{1} = \delta_1$, that is, $\overline{\alpha}_{1} = \pi_1$. We have $\Phi(\overline{\alpha}) = \Phi(\Phi^{-1}(\lambda, \mu)) = (\lambda, \mu)$. Let $\overline{\alpha}^{(1)}$ denote the permutation obtained by removing $\pi_1$ from $\overline{\alpha}$. Then   $\des(\overline{\alpha}^{(1)})=d-1$ or $\des(\overline{\alpha}^{(1)})=d$. 
 \medskip 
 
 {\it Case 1.1'}. If $\des(\overline{\alpha}^{(1)})=d$, then $\des(\overline{\alpha}^{(0)})=\des(\overline{\alpha}^{(1)})=d$, and so $j_1=1$,  and $d_1(\pi)=s$. Since $\overline{\alpha}_{1} = \pi_1$, we derive that  
 $\maj(\overline{\alpha}^{(0)})-\maj(\overline{\alpha}^{(1)})=d.$ 
Thus, 
\[\mu_1=t(1)=\maj(\overline{\alpha}^{(0)})-\maj(\overline{\alpha}^{(1)})-d_1(\pi)=d-s.\] 

 {\it Case 1.2'}. If $\des(\overline{\alpha}^{(1)})=d-1$, then by the definition of the map $\Phi$, we see that  $j_1>1$, and $\des(\overline{\alpha}^{(j_1)})=d-j_1+1$.  Since $\overline{\alpha}_{1} = \pi_1$ and all the elements of $\delta$ are larger than the elements of $\pi$, it follows that ${\overline{\alpha}}^{(j_1-1)}_{1} = \pi_{j_1}$. Thus,   
$\maj(\overline{\alpha}^{(j_1-1)})-\maj(\overline{\alpha}^{(j_1)})=d-j_1+1
$ and $d_{j_1}(\pi)=s-j_1+1$. Consequently,  \[\mu_1=t(j_1)=\maj(\alpha^{(j_1-1)})-\maj(\alpha^{(j_1)})-d_{j_1}(\pi)=d-s.\] 

In both  cases, we derive that $\mu_1=d-s$, which  contradicts the condition that   $\mu_{1} \leq d-s-1$. Therefore, the assumption is false, so $\overline{\alpha}_{1} = \delta_1$, which implies that $\overline{\alpha} \in \pi \shuffle_l \delta$.  

\medskip

(2)  Let $\alpha \in \pi \shuffle_{ls} \delta$, where $\des(\alpha)=d$. Assume that $\Phi(\alpha) = (\lambda, \mu)$. The proof of  (2) in this theorem is equivalent to showing that 
\begin{equation}\label{part:case2}
m >\lambda_1 \geq \cdots \geq \lambda_{d - r} \geq d - s> \mu_1 \geq \cdots \geq \mu_{n - d + r} \geq 0.
\end{equation}
Observe that if $\alpha \in \pi \shuffle_{ls} \delta$, then $\alpha \in \pi \shuffle_{l} \delta$. 
According to (1) in this theorem, we see that  
\[
m \geq \lambda_1 \geq \cdots \geq \lambda_{d - r} \geq d - s >\mu_1 \geq \cdots \geq \mu_{n - d + r} \geq 0.
\]
 It remains to show that $\lambda_{1} < m$ if $\lambda \neq \emptyset$.  

Since $\alpha \in \pi \shuffle_{ls} \delta$, we have $\alpha_1=\delta_1$ and $\alpha_{n+m}=\delta_m$.  
 Recall that $\alpha^{(i)}$ is the permutation obtained by removing $\pi_1,\ldots, \pi_i$ from $\alpha$. Observe that $\alpha^{(n)}=\delta$ and $\des(\delta)=r$, so $\des(\alpha^{(n-1)})=r+1$ or $\des(\alpha^{(n-1)})=r$. 

 \medskip 
 
 {\it Case 2.1}. If $\des(\alpha^{(n-1)})=r+1$, then  $i_{d-r}=n$,  and $d_n(\pi)=0$. By definition, it is easy to see that $\operatorname{im}(\alpha^{(n)}, m, \pi_n) = m. $ 
 Since $\alpha_{n+m}=\delta_m$, it implies that $\pi_n$ could not be inserted in the last position, by Proposition \ref{prop:insert}, we derive that  $\operatorname{im}(\alpha^{(n)}, i, \pi_n)\leq m-1$ for $0\leq i<m$. Hence 
 \begin{equation}\label{pf:inserta}
 \maj(\alpha^{(n-1)})-\maj(\alpha^{(n)})\leq m-1.
 \end{equation}
It follows that 
\[\lambda_1=t(i_{d-r})=\maj(\alpha^{(n-1)})-\maj(\alpha^{(n)})-d_n(\pi)\leq m-1.
\]

 {\it Case 2.2}. If $\des(\alpha^{(n-1)})=r$, then by the definition of the map $\Phi$, we see that  $i_{d-r}<n$. Moreover, $\des(\alpha^{(i_{d-r}-1)})=r+1$ and $\des(\alpha^{(l)})=r$ for $i_{d-r}\leq l\leq n$. According to Proposition \ref{prop:insertb}  and by \eqref{pf:inserta}, we derive that 
 \[\maj(\alpha^{(i_{d-r}-1)})-\maj(\alpha^{(i_{d-r})})\leq m-1+d_{i_{d-r}}(\pi)\]
It follows that 
\[\lambda_1=t(i_{d-r})=\maj(\alpha^{(i_{d-r}-1)})-\maj(\alpha^{(i_{d-r})})-d_{i_{d-r}}(\pi)\leq m-1.
\]
     
\medskip 

Conversely, let $(\lambda, \mu)$ be a pair of partitions satisfying \eqref{part:case2}. Assume that $\Phi^{-1}(\lambda, \mu) = \overline{\alpha}$, where the map $\Phi^{-1}$ is the inverse of $\Phi$. In light of the first result in this theorem, we derive that $\overline{\alpha}=\overline{\alpha}_1\cdots \overline{\alpha}_{n+m} \in \pi \shuffle_l \delta$. To prove that $\overline{\alpha} \in \pi \shuffle_{ls} \delta$, it suffices to show that $\overline{\alpha}_{n+m} = \delta_m$. Suppose to the contrary that $\overline{\alpha}_{n+m} = \delta_m$ and assume that $\overline{\alpha}_{n+m}  = \pi_n$. We have $\Phi(\overline{\alpha}) = \Phi(\Phi^{-1}(\lambda, \mu)) = (\lambda, \mu)$. Let $\overline{\alpha}^{(i)}$ denote the permutation obtained by removing $\pi_1,\ldots, \pi_i$ from $\overline{\alpha}$. Since $\overline{\alpha}_{n+m}  = \pi_n$ and all of the elements of $\delta$ are larger than the elements of $\pi$, we derive that  $\des(\overline{\alpha}^{(n-1)})=r+1$. In this case, we see that $i_{d-r}=n$,  and $d_n(\pi)=0$. Since $\overline{\alpha}_{n+m} = \pi_n$, we derive that  
 $\maj(\overline{\alpha}^{(n-1)})-\maj(\overline{\alpha}^{(n)})=m.$ 
Thus, 
\[\lambda_1=t(i_{d-r})=\maj(\overline{\alpha}^{(n-1)})-\maj(\overline{\alpha}^{(n)})-d_{n}(\pi)=m,\] 
 which  contradicts the condition that   $\lambda_{1} <m$. Therefore, the assumption is false, so $\overline{\alpha}_{m+n} = \delta_m$, which implies that $\alpha \in \pi \shuffle_{ls} \delta$. 

 \medskip 

(3) Let $\alpha \in \pi \shuffle_{ll} \delta$, where $\des(\alpha)=d$. Assume that $\Phi(\alpha) = (\lambda, \mu)$. To establish (3) in this theorem, by means of \eqref{int-GassCoeft}, it is enough to show that
\begin{equation} \label{part:case3}
m \geq \lambda_1 \geq \cdots \geq \lambda_{d - r} > d - s> \mu_1 \geq \cdots \geq \mu_{n - d + r} \geq 0.
\end{equation}
Note that if $\alpha \in \pi \shuffle_{ll} \delta$, then $\alpha \in \pi \shuffle_{l} \delta$. From the first part of this theorem, we derive that 
\[
m \geq \lambda_1 \geq \cdots \geq \lambda_{d - r} \geq d - s >\mu_1 \geq \cdots \geq \mu_{n - d + r} \geq 0.
\]
We proceed to show that $\lambda_{d - r} >d-s$ if $\lambda \neq \emptyset$.

Since $\alpha \in \pi \shuffle_{ll} \delta$, we have $\alpha_1=\delta_1$ and $\alpha_2=\delta_2$.  
 Recall that $\alpha^{(1)}$ is the permutation obtained by removing $\pi_1$ from $\alpha$. Then $\des(\alpha^{(1)})=d-1$ or $\des(\alpha^{(1)})=d$. 

\medskip 
 
 {\it Case 3.1}. If $\des(\alpha^{(1)})=d-1$, then by the definition of the map $\Phi$, we see that  $i_1=1$ and $d_1(\pi)=s$. 
 Since $\alpha_1=\delta_1<\alpha_2=\delta_2$ , it implies that $\pi_1$ could not be inserted in the first position and the second position. Moreover, $\operatorname{im}(\alpha^{(1)}, 0, \pi_1) = d-1$ and $
\operatorname{im}(\alpha^{(1)}, 1, \pi_1) = d. $ Thus,  by Proposition \ref{prop:insert}, we derive that
 $\maj(\alpha^{(0)})-\maj(\alpha^{(1)})>d,$ as $\des(\alpha^{(0)})=\des(\alpha^{(1)})+1$. 
 It follows that 
 \[\lambda_{d-r}=t(i_1)=\maj(\alpha^{(0)})-\maj(\alpha^{(1)})-d_1(\pi)> d-s.\]

{\it Case 3.2}.  If $\des(\alpha^{(1)})=d$,    then by the definition of the map $\Phi$, we see that  $i_{1}>1$. Moreover,  $\des(\alpha^{(i_{1})})=d-1$ and $\des(\alpha^{(l)})=d$ for $1\leq l\leq i_1-1$.  Since $\alpha_1=\delta_1<\alpha_2=\delta_2$ , it implies that $\pi_{i_1}$ could not be inserted in the first position and the second position of $\alpha^{(i_1)}$. Moreover, $\alpha^{(i_1)}_1=\delta_1$ and $\alpha^{(i_1)}_1=\delta_2$, thus, we have 
$\operatorname{im}(\alpha^{(i_1)}, 0, \pi_{i_1}) = d-1$ and $ 
\operatorname{im}(\alpha^{(i_1)}, 1, \pi_1) = d.  $ 
Hence  $\maj(\alpha^{(i_1-1)})-\maj(\alpha^{(i_1)})>d$ since $\des(\alpha^{(i_{1}-1)})=\des(\alpha^{(i_{1})})+1$. Observe that  $d_{i_1}(\pi)\leq \des(\pi)=s$, it follows that  
\[\mu_1=t(i_1)=\maj(\alpha^{(i_1-1)})-\maj(\alpha^{(i_1)})-d_{i_1}(\pi)\geq  d+1-s.\]

\medskip 

 Conversely, let $(\lambda, \mu)$ be a pair of partitions satisfying \eqref{part:case3}. Assume that $\Phi^{-1}(\lambda, \mu) = \overline{\alpha}$, where the map $\Phi^{-1}$ is the inverse of $\Phi$. In light of the first part of this theorem, we derive that $\overline{\alpha} = \overline{\alpha}_1 \cdots \overline{\alpha}_{n + m} \in \pi \shuffle_l \delta$. To prove that $\overline{\alpha} \in \pi \shuffle_{ll} \delta$, it suffices to show that $\overline{\alpha}_{2} = \delta_2$. Suppose to the contrary  and assume that $\overline{\alpha}_{2} = \pi_1$. We have $\Phi(\overline{\alpha}) = \Phi(\Phi^{-1}(\lambda, \mu)) = (\lambda, \mu)$. Let $\overline{\alpha}^{(1)}$ denote the permutation obtained by removing $\pi_1$ from $\overline{\alpha}$. Then   $\des(\overline{\alpha}^{(1)})=d-1$ or $\des(\overline{\alpha}^{(1)})=d$.

 \medskip  

 {\it Case 3.1'}. If $\des(\overline{\alpha}^{(1)})=d-1$, then by the definition of the map $\Phi$, we see that  $i_1=1$.    Since $\overline{\alpha}_{2} = \pi_1$, it follows that    
$\maj(\overline{\alpha}^{(0)})-\maj(\overline{\alpha}^{(1)})=d.
$
  Consequently,  \[\lambda_{d-r}=t(1)=\maj(\alpha^{(0)})-\maj(\alpha^{(1)})-d_{1}(\pi)=d-s.\] 

 {\it Case 3.2'}. If $\des(\overline{\alpha}^{(1)})=d$, then   by the definition of the map $\Phi$, we see that  $i_{1}>1$, and $\des(\overline{\alpha}^{(i_{1})})=d-1$ and $\des(\overline{\alpha}^{(l)})=d$ for $1\leq l\leq i_1-1$.  Since $\overline{\alpha}_1=\delta_1$ and $\overline{\alpha}_2=\pi_1$, it implies that $\pi_{i_1}$ should  be inserted in    the second position of $\alpha^{(i_1)}$, that is, $\overline{\alpha}^{(i_1-1)}_1=\delta_{1}$ and $\overline{\alpha}^{(i_1-1)}_2=\pi_{i_1}$,  
 otherwise, we could not get $\overline{\alpha}^{(i_1-1)}, \ldots, \overline{\alpha}^{(1)}$ so that  
 $\des(\overline{\alpha}^{(l)})=d$ for $1\leq l\leq i_1-1$ and $\overline{\alpha}_2=\pi_1$ since $\delta_1<\delta_2$. Moreover, $d_{i_1}(\pi)=s$.    
Hence,   
 \[\mu_1=t(i_1)=\maj(\overline{\alpha}^{(i_1-1)})-\maj(\overline{\alpha}^{(i_1)})-d_{i_1}(\pi)= d-s.\]

In both  cases, we derive that $\lambda_{d-r}=d-s$, which  contradicts the condition that   $\lambda_{d-r}>d-s$. Therefore, the assumption is false, so $\overline{\alpha}_{2} = \delta_2$, which implies that $\overline{\alpha} \in \pi \shuffle_{ll} \delta$.

\section{simsun  permutations and Andr\'e II permutations}

This section is dedicated to  establishing   relation (b) in Theorem \ref{th:equi-shape} by constructing a bijection between the set of $n$-simsun  permutations and the set of $n$-Andr\'e II permutations  with the same tree shape $T$.

\medskip 

\begin{Theorem} \label{thm:bij} For   any  unlabeled binary tree $T$ in $\URL_n$, there exists a bijection   $\Omega$ between the set  of $n$-simsun  permutations in $\RS(T)$  and the set    of $n$-Andr\'e II permutations in $\AndII(T)$.
\end{Theorem}
\begin{proof}
  Given   an unlabeled binary tree $T$ in $\URL_n$, let $\sigma $ be a simsun  permutation in $\RS(T)$, we aim to define $\tau=\Omega(\sigma) $ belonging to  $\AndII(T)$.

Let $\hat{T}:=\Psi(\sigma)$ be the increasing binary tree corresponding to $\sigma$ under the bijection $\Psi$ defined in Definition \ref{defi:mappsi} and let 
\begin{equation} \label{rel-Ta}
R_{\hat{T}}=\{v_0<v_1<\cdots <v_{m-1}<v_{m}\}
\end{equation}
be the set of vertices that don't belong to any left subtrees of $\hat{T}$. By the definition of simsun  tree, we see that 
 $v_0=1$ and $v_m=n$. Let $\bar{R}_{\hat{T}}$ be the set of vertices that don't belong to $R_{\hat{T}}$. Assume that 
 \begin{equation} \label{rel-Tb}
 \bar{R}_{\hat{T}}=\{s_1<\cdots <s_{n-m-1}\}.
 \end{equation}
Note that 
$R_{\hat{T}} \cup \bar{R}_{\hat{T}}=\{1,2,\ldots, n\}.$ 
 
 We then relabel the elements of $R_{\hat{T}}$ according to the permutation 
\[\left( \begin{array} {ccccccc}
v_0=1&v_1&v_2 &v_3&\cdots & n\\
1&v_0+1=2&v_1+1&v_2+1 &\cdots & v_{m-1}+1
\end{array} 
\right).
\]

For other vertices that belong to $\bar{R}_{\hat{T}}$, we just add one to each of their values. 
By this operation, we obtain a new binary tree  $\widetilde{T}$.  From the above construction, we see that 
\begin{equation} \label{rel-Tc}
R_{\widetilde{T}}=\{ 1<2<v_1+1<\cdots <v_{m-1}+1\},
\end{equation}
and 
\begin{equation} \label{rel-Td}
\bar{R}_{\widetilde{T}}=\{s_1+1<\cdots <s_{n-m-1}+1\}.
\end{equation}
It is straightforward  to check that $ R_{\widetilde{T}} \cup \bar{R}_{\widetilde{T}}=\{1,2,\ldots, n\}$ and $\widetilde{T}$ is an increasing binary tree sharing the same tree shape as $\hat{T}$. Then $\tau$ is defined to be the permutation generated by the increasing binary tree  $\widetilde{T}$ under the bijection $\Psi$, that is, $\tau=\Psi^{-1}(\widetilde{T})$. 

\medskip 

 For example, for the simsun  permutation $\sigma= 21473658$, the corresponding increasing binary tree is shown in Fig. \ref{fig:bije} (a). We have $R_{\hat{T}}=\{1,3,5, 8\}$. Relabel the elements of $R_{\hat{T}}=\{1,3,5,8\}$ according to the permutation 
\[\left( \begin{array} {cccc}
1&3&5 &8 \\
1&2&4&6 
\end{array} 
\right)
\]
and add one to each of the vertices not belong to $R_{\hat{T}}$. We obtain the increasing binary tree $\widetilde{T}$ shown on  Fig. \ref{fig:bije} (b).  
 Then    $\tau=\Psi^{-1}(\widetilde{T})=31582746$.  
\begin{figure}
\centering
\caption*{ (a) simsun  tree $\hat{T}$ \hspace{3cm} (b) Andr\'e II tree $\widetilde{T}$}
\begin{tikzpicture}
[sibling distance=15mm, level distance=8mm,
every node/.style={circle, draw, fill=white, inner sep=0.5mm},
level 1/.style={sibling distance=17mm},
level 2/.style={sibling distance=14mm},
level 3/.style={sibling distance=8mm}]
\node [circle, draw=red, fill=white, inner sep=0.5mm] {1}
child { node [circle, draw, fill=white, inner sep=0.5mm] {2}
child {     edge from parent [white] } child {     edge from parent [white] } } child { node [circle, draw=red, fill=white, inner sep=0.5mm] {3}
child { node [circle, draw, fill=white, inner sep=0.5mm] {4}
child {     edge from parent [white] } child { node [circle, draw, fill=white, inner sep=0.5mm] {7}
 } } child { node [circle, draw=red, fill=white, inner sep=0.5mm] {5}
child { node [circle, draw, fill=white, inner sep=0.5mm] {6}
} child {  node [circle, draw=red, fill=white, inner sep=0.5mm] {8}   } } };

\node[draw=none,fill=none] at (0.5,-3.5) {$\sigma=2\, 1\, 4\, 7\, 3\, 6\, 5\, 8 $} ;
\draw[->, line width=1pt, shorten >=2pt, shorten <=2pt]  (2.5,-1.3) --(4.5,-1.3)
              node[draw=none,fill=none, yshift=4mm,xshift=-10mm] {$\Omega$};
              \draw[->, line width=1pt, shorten >=2pt, shorten <=2pt]  (4.5,-1.7) --(2.5,-1.7)
              node[draw=none,fill=none, yshift=-4mm, xshift=10mm] {$\Omega^{-1}$};
\end{tikzpicture}
\begin{tikzpicture}
[sibling distance=15mm, level distance=8mm,
every node/.style={circle, draw, fill=white, inner sep=0.5mm},
level 1/.style={sibling distance=17mm},
level 2/.style={sibling distance=14mm},
level 3/.style={sibling distance=8mm}]
\node [circle, draw=red, fill=white, inner sep=0.5mm] {1}
child { node [circle, draw, fill=white, inner sep=0.5mm] {3}
child {     edge from parent [white] } child {     edge from parent [white] } } child { node [circle, draw=red, fill=white, inner sep=0.5mm] {2}
child { node [circle, draw, fill=white, inner sep=0.5mm] {5}
child {     edge from parent [white] } child { node [circle, draw, fill=white, inner sep=0.5mm] {8}
 } } child { node [circle, draw=red, fill=white, inner sep=0.5mm] {4}
child { node [circle, draw, fill=white, inner sep=0.5mm] {7}
} child {  node [circle, draw=red, fill=white, inner sep=0.5mm] {6} } } };
\node[draw=none,fill=none] at (0.5,-3.5) {$\tau=3\, 1\, 5\, 8\, 2\, 7\, 4\, 6 $} ;
\end{tikzpicture}
\caption{An illustration of the bijection $\Omega$.} \label{fig:bije}
\end{figure}

\medskip 

We proceed to show that  $\tau \in \AndII (T)$. It suffices to show that  $\widetilde{T}$  satisfies the property of Andr\'e II trees: for any internal vertex  ${s}$ in   $\widetilde{T}$, the right subtree $\widetilde{T}_r(s)$ contains the vertex with
the minimum  label in $\widetilde{T}(s)$  excluding $s$ itself. By definition, the minima of an empty subtree is defined as $+\infty$.

Since $\hat{T}$ is a simsun  tree, by definition, for  any internal vertex  $s \in \bar{R}_{\hat{T}}$, the minimum label in its left subtree is   always larger than the minimum label in its right subtree. Otherwise,  removing vertices with labels larger than this minimum label would result in a tree violating the simsun  tree property. Thus, from the above construction, we see that  for any internal vertex  $s \in \bar{R}_{\widetilde{T}}$,  the right subtree $\widetilde{T}_r(s)$ contains the vertex with
the minimum label in $\widetilde{T}(s)$  excluding $s$ itself. 

 For any internal vertex  ${s} \in {R}_{\widetilde{T}}$, we claim that the minimum label in its left subtree ($\widetilde{T}_l(s)$) is   always larger than the minimum label in its right subtree ($\widetilde{T}_r(s)$). This holds because, in the increasing binary tree $\hat{T}$,   the vertex with label equal to the minimum label of $\widetilde{T}_r(s)$ minus one is the parent of the vertex with label equal to  the minimum label minus one in  $\widetilde{T}_l(s)$. Consequently, the minimum label in  $\widetilde{T}_l(s)$ is  always larger than   that in $\widetilde{T}_r(s)$, proving our claim. Hence for any internal vertex  $s \in {R}_{\widetilde{T}}$,  the right subtree $\widetilde{T}_r(s)$ contains the vertex with
the minimum label in $\widetilde{T}(s)$  excluding $s$ itself. 

Thus,  $\widetilde{T}$ is an Andr\'e II tree, so $\tau=\Psi^{-1}(\widetilde{T})$ is an Andr\'e II permutation with  tree shape $T$. Moreover, it is easy to verify that  the above procedure is reversible.  Hence the map $\Omega$ is a bijection between $\RS(T)$ and $\AndII(T)$ for any given unlabeled binary tree $T$ in $\URL_n$.
\end{proof}
 
\medskip 

{\Remark  The bijection $\Omega$ in Theorem \ref{thm:bij} is constructed in the spirit of Sc\"utzenberger's jeu de taquin \cite{Sch-1972}. Lin and Kim \cite{Lin-Kim-2021} constructed a similar bijection between 0-1-2-increasing trees and binary increasing trees, which yields a bijection  between 000-avoiding inversion sequences and Simsun permutations. }

\medskip 

From Theorem \ref{thm:bij} and Proposition \ref{th:equi-shape-des}, we see that the bijection $\Omega$   preserves $\des$  and $\maj$ statistics. To  establish   relation (b) in Theorem \ref{th:equi-shape}, it remains to show that the bijection $\Omega$   preserves  $\ides$  statistics. 
In fact, this bijection $\Omega$ also induces  relations involving the inversions ($\inv$), the imajor index ($\imaj$) and  right-to-left minima ($\RLmin$)  between simsun  permutations and Andr\'e II permutations (see Proposition \ref{th:equi-shape-des2}).   Recall that the number of inversions of $\sigma=\sigma_1\cdots \sigma_n$, denoted $\inv(\sigma)$ is the count of pairs of indices $(i, j)$ where $1\leq i < j\leq n$ and $\sigma_i > \sigma_j$. The number of  right-to-left minima, denoted $\RLmin(\sigma)$ is the count of the elements $\sigma_i$ such that $\sigma_j > \sigma_i$  for every
 $j>i$. Let ${\IDes}(\sigma)$ denote the set of descents of  the inverse permutation $\sigma^{-1}$, that is, ${\IDes}(\sigma)={\Des}(\sigma^{-1})$.
 The  imajor index of $\sigma$, denoted $\imaj(\sigma)$  is the sum of  the indices in  ${\IDes}(\sigma)$.   
 \medskip 

Before proceeding, let us first define an inversion and an idescent of an increasing binary tree.

\begin{Definition}\label{defi-inv-ides-Andre} Let $T$ be an increasing binary  tree on the set $[n]$. An inversion of $T$ is a pair of vertices $(i,j)$, where $i>j$,   and either $j$ lies to the right of the path from root 1 to $i$ or $j$ is on the path from root 1 to $i$ and left child of $j$ is contained in this path.  Moreover, $i-1$ is called an idescent of $T$ if  $(i,i-1)$ is an inversion of $T$. 
\end{Definition}

The number of  inversions of $T$ is denoted by ${\rm inv}\  (T)$,  the set of idescents by $\IDes(T)$ and  the set of vertices of ${T}$ that don't belong to any left subtrees of ${T}$ by $R_{{T}}$. For example, for the increasing binary tree $\hat{T}$  depicted in  Fig.~\ref{fig:bije}~(a),  we see that $(2,1),\,(4,3),\, (7,3),\, (7,6),\, (7,5),\, (6,5)$ are the inversions of $\hat{T}$, whereas, $1,\,3,\,6,\, 5$ are idescents of $\hat{T}$. Thus, $\inv (\hat{T})=6$,   $\IDes(\hat{T})=\{1,\,3,\,5,\, 6\}$ and $R_{\hat{T}}=\{1,3,5,8\}$. 

\medskip 

From the definition of $\Psi$, it is not difficult to derive the following proposition.

\begin{Proposition} \label{th:equi-shape-des2} Let $\sigma$ be a permutation and let $T_\sigma=\Psi(\sigma)$ be the increasing binary tree corresponding to $\sigma$ under the bijection $\Psi$.  Then 
$$
{\inv}(\sigma) = { \inv}(T_\sigma), \quad \RLmin(\sigma)=|R_{T_\sigma}|, \quad 
{\IDes}(\sigma) = {\rm \IDes}(T_\sigma). 
$$
\end{Proposition}

We conclude this paper with the proof of the following proposition, which directly implies relation (b) in Theorem \ref{th:equi-shape}.

\begin{Proposition}\label{simsun _tree_Andre2_tree_bijection}  Given an  unlabeled binary tree $T \in \URL_n$, let $\sigma \in \RS(T)$ be a simsun permutation   and let $\tau=\Omega(\sigma) \in \AndII(T)$ be  the corresponding Andr\'e II permutation.  
Then  
\begin{align}
\ides(\tau) &=\ides(\sigma), \label{bij:rel}\\  
\imaj(\tau) &=\imaj(\sigma)+\ides(\sigma), \label{bij:relb} \\
\inv(\tau) &=\inv(\sigma)+n-1-\RLmin(\sigma). \label{bij:relc}
\end{align}
\end{Proposition} 

\begin{proof} Let $T$ be an increasing binary tree.  Recall that $R_{{T}}$ is the set of vertices of ${T}$ that don't belong to any left subtrees of ${T}$. Let  $\bar{R}_T$ be the set of vertices that don't belong to $R_T$.  By  definition, we see that for any  $i\in R_{{T}}$,  there doesn't exist $j$ such that $(i,j)$ is an inversion of ${{T}}$.   Let   
\begin{align*}
A({T}): =&\{ i-1 \colon  (i,i-1) \text{ is an inversion, where }   i\in \bar{R}_{{T}}, i-1\in \bar{R}_{{T}} \},\\
B({T}): =&\{ i-1 \colon (i,i-1) \text{ is an inversion, where } i\in \bar{R}_{{T}}, i-1\in R_{{T}} \},\\
C({T}): =& \{ (i,j) \colon (i,j) \text{ is an inversion of } T, \text{ where }i\in \bar{R}_{{T}}, j\in \bar{R}_{{T}} \}, \\
D({T}): =& \{ (i,j) \colon (i,j) \text{ is an inversion of } T, \text{ where }i\in \bar{R}_{{T}}, j\in {R}_{{T}} \}.
\end{align*}

Suppose that $\hat{T}$ is  the simsun  tree corresponding to $\sigma$ and    $\widetilde{T}$ is  the Andr\'e II tree  corresponding    to $\tau$, that is   $\hat{T}=\Psi^{-1}(\sigma)$ and   $\widetilde{T}=\Psi^{-1}(\tau)$.  We then have  
\begin{align*}\label{eq:ides_0}
&\IDes(\hat{T}) =A(\hat{T}) \cup B(\hat{T}) \quad \text{and} \quad  \inv(\hat{T}) =|C(\hat{T})|+|D(\hat{T})|  
\end{align*} 
and 
\[
\IDes(\widetilde{T})=A(\widetilde{T}) \cup B(\widetilde{T})\quad  \text{ and }      \quad\inv(\widetilde{T}) =|C(\widetilde{T})|+|D(\widetilde{T})|.
 \]
 
From relations \eqref{rel-Ta}, \eqref{rel-Tb}, \eqref{rel-Tc} and \eqref{rel-Td} in the construction of $\Phi$, it is straightforward to derive that  for $i\in [n]$, $
i \in A(\hat{T}) \Leftrightarrow    i+1 \in A(\widetilde{T})$ and $\quad i \in B(\hat{T}) \Leftrightarrow    i+1 \in B(\widetilde{T}) $. It follows that for $i\in [n]$, 
\begin{align*}
i \in \IDes(\hat{T}) \Leftrightarrow    i+1 \in \IDes(\widetilde{T}). 
\end{align*} 
Hence, by Proposition \ref{th:equi-shape-des2}, we obtain  \eqref{bij:rel}
and \eqref{bij:relb}.

Invoking relations \eqref{rel-Ta}, \eqref{rel-Tb}, \eqref{rel-Tc} and \eqref{rel-Td} again, we derive that  for $i\in [n]$ and $j \in [n]$, 
\begin{align}\label{eq:inv_3}
(i,j) \in C(\hat{T}) \Leftrightarrow   (i+1,j+1) \in C(\widetilde{T}).
\end{align}
On the other hand, for any $i\in \bar{R}_{\hat{T}}$, by the construction of the bijection $\Phi$,  we see that $i+1\in \bar{R}_{\widetilde{T}}$. For a given $i\in \bar{R}_{\hat{T}}$, 
 an inversion  $(i,j)$ in $\hat{T}$  with $j\in {R}_{\hat{T}} $ if and only if $(i+1,j+1)$ is an inversion in ${\widetilde{T}}$ with $j+1\in {R}_{{\widetilde{T}}}$, and $i+1$ is not in the left subtree of $j+1$. Additionally, for each $i\in \bar{R}_{\hat{T}}$, there exists a $j\in {R}_{\widetilde{T}}$ such that  $i+1$ lies in the left subtree of $j$ and $(i+1,j)$   is an inversion in $\widetilde{T}$.  Thus, for a given $i\in \bar{R}_{\hat{T}}$, 
\begin{align*}
&\mid\{j \in {R}_{\widetilde{T}} :  (i+1,j) \text{ is an inversion in } \widetilde{T} \}\mid \nonumber
\\&=
\mid\{j \in {R}_{\hat{T}} :  (i,j) \text{ is an inversion in } \hat{T} \}\mid + 1.
\end{align*}
It follows that 
\[|D(\widetilde{T})|=|D(\hat{T})|+|  \bar{R}_{T} |,\]
Combining this with \eqref{eq:inv_3}, we obtain 
\[\inv(\widetilde{T}) =\inv(\hat{T}) +|\bar{R}_{T}|, 
\]
and by Proposition \ref{th:equi-shape-des2}, we derive that \eqref{bij:relc}. This completes the proof. 
\end{proof}

 \vskip 0.2cm
\noindent{\bf Acknowledgment.}   We would like to express our gratitude to Zhicong Lin for bringing his paper with Dongsu Kim   to our attention. This work
was supported by the National Natural Science Foundation of China.

\bibliographystyle{plain} 

\end{document}